\documentclass[10pt,a4paper]{article}
\usepackage{hyperref}
\usepackage{a4wide}
\usepackage{graphicx,xcolor,textpos}
\usepackage{helvet}
\usepackage[utf8]{inputenc}
\usepackage[T1]{fontenc}
\usepackage[english]{babel}
\usepackage{amsmath}
\usepackage{amsfonts}
\usepackage{amssymb,amsthm}
\usepackage{graphicx}
\usepackage{mathtools}
\usepackage{bbold}
\usepackage{csquotes}
\usepackage{tikz-cd}
\usepackage{cases}
\usepackage{listings}
\usepackage{algorithm}
\usepackage{algorithmic}
\usepackage{enumerate}   

\newtheorem{thm}{Theorem}[section]
\newtheorem{thmx}{Theorem}

\newtheorem{conj}{Conjecture}
\newtheorem{lemma}[thm]{Lemma}
\newtheorem{prop}[thm]{Proposition}

\newtheorem{cor}[thm]{Corollary}
\theoremstyle{definition}
\newtheorem{df}[thm]{Definition}
\newtheorem{nota}[thm]{Notation}

\newtheorem{ques}{Question}
\theoremstyle{remark}
\newtheorem{rem}{Remark}
\theoremstyle{definition}
\newtheorem{ex}[thm]{Example}
\theoremstyle{definition}
\newcommand{\Z}{\mathbb{Z}}
\newcommand{\N}{\mathbb{N}}
\newcommand{\Q}{\mathbb{Q}}
\newcommand{\R}{\mathbb{R}}
\newcommand{\C}{\mathbb{C}}
\newcommand{\F}{\mathbb{F}}

\newcommand{\norm}[1]{\|#1\|}

\DeclareMathOperator{\Aut}{Aut}

\DeclareMathOperator{\Gal}{Gal}
\DeclareMathOperator{\lcm}{lcm}

\DeclareMathOperator{\GL}{GL}
\DeclareMathOperator{\RF}{RF}
\DeclareMathOperator{\Id}{Id}
\DeclareMathOperator{\BS}{BS}
\DeclareMathOperator{\Fit}{Fit}
\DeclareMathOperator{\degr}{deg}

\title{Residual Finiteness Growth in Minimax Groups}
\author{Jonas Der\'e and Joren Matthys\thanks{KU Leuven Campus Kulak Kortrijk, Department of Mathematics, Research unit `Algebraic Topology and Group Theory', B-8560 Kortrijk, Belgium. The authors were supported by Internal Funds KU Leuven (project number 3E220559).}}
\date{}
\begin{document}
	\maketitle
	\begin{abstract}
	If $g\in G$ is a non-trivial element in a residually finite group, then there exists by definition a finite group $Q$ and a homomorphism $\varphi: G \to Q$ such that $\varphi(g) \neq e$. The residual finiteness growth $\RF_G$ of a finitely generated residually finite group $G$ estimates the size of $Q$ in terms of the word norm $\norm{g}$ of the element $g\in G$. This function has been studied for several classes of groups, including free groups, lamplighter groups and nilpotent groups.
	
	For finitely generated linear groups $G\leq \GL(m, \C)$ this function is known to be bounded by $\RF_G(r) \preceq r^{m^2+1}$, which is quadratic in $m$. This paper establishes an improved bound of the form $\RF_G(r) \preceq r^{4k}$ with $k$ the Pr\"ufer rank of $G$ for certain virtually solvable linear groups, namely minimax groups, a class which includes virtually polycyclic and Baumslag-Solitar groups. Moreover, the upper bound is invariant under taking finite extensions, and also establishes an improved polylogarithmic version for virtually nilpotent groups, generalizing the known exact bound for virtually abelian groups. If the group is not virtually nilpotent, we prove that $\RF_G(r)$ is at least linear, improving a recent result.
	\end{abstract}
	\section{Introduction}
	Let $G$ be a finitely generated residually finite group. By definition, there exists for every non-trivial element $e\neq g \in G$ a homomorphism $\varphi: G \to Q$ to a finite group $Q$ such that $\varphi(g) \neq e$. Since the initial paper \cite{bou2010quantifying} by Bou-Rabee, numerous papers have appeared that bound the size of $Q$ in terms of the word norm $\norm{g}$ of the element $g\in G$ for several classes of groups $G$. This bound on $|Q|$ is encapsulated into the residual finiteness growth $\RF_G: \N \to \N$: it is the minimal function such that if $\norm{g} \leq r$, then $Q$ exists as above with $|Q| \leq \RF_G(r)$. The residual finiteness growth has been studied for several classes of groups, including virtually abelian groups \cite{math_virt_ab},  free groups \cite{bradford2019short}, certain branch groups \cite{branch_groups}, lamplighter groups \cite{bou2019residual}, \ldots \, The survey article \cite{survey2022} states several known results and open questions about this function.
	
	In this paper, we focus on a subclass of linear groups $G\leq \GL(m, \F)$ over fields $\F$ of characteristic zero, namely all finitely generated residually finite virtually solvable minimax groups, as introduced in Section \ref{sec_minimax}. In this paper, we will show that this class can be characterized in the following way:
	\begin{thmx} \label{thm_intro_minimax}
		A finitely generated group $G$ is a residually finite virtually solvable minimax group if and only if $G$ fits in a short exact sequence of the form
		\begin{equation} \label{eq_intro_seq_minimax}
			1 \to K \to G \to H \to 1,
		\end{equation}
		where $K$ is a torsion-free nilpotent group of finite Pr\"ufer rank and $H$ is virtually abelian.
	\end{thmx}
	Throughout this article, we will call these groups $\mathcal{M}$-groups. This class includes the virtually polycyclic groups and also the Baumslag-Solitar groups $\BS(1,n)$. We prove the following bound:
	\begin{thmx} \label{thm_intro_upper}
		Let $G$ be a $\mathcal{M}$-group with corresponding short exact sequence as in Equation \eqref{eq_intro_seq_minimax}. If $m$ is the Pr\"ufer rank of $K$ and $n$ the maximal rank of a free abelian subgroup of $H$, then,
		$$\RF_G \preceq r^{m + 4n}.$$
	\end{thmx}
\noindent Recall that if $G \leq \GL(k, \F)$ is a finitely generated linear group over a field $\F$ of characteristic zero, then $\RF_G$ is bounded above by the polynomial $r^{k^2-1}$ by \cite{franz2017quantifying}. For several subclasses of $\mathcal{M}$-groups, the bound of Theorem \ref{thm_intro_upper} is the first that does not use the linearity of the groups, providing two different improvements over the classical bound. 
	
	Firstly, the bound $r^{k^2-1}$ grows quadratically in the dimension $k$, whereas the new bound grows linearly in $m$ and $n$. As there exist even finitely generated nilpotent groups where the minimal $k$ grows linearly in $m$, the new bound is sharper for many $\mathcal{M}$-groups, see remark \ref{ex:motivation}. Note that this bound also allows to give an upper bound for virtually polycyclic groups $G$ in terms of their Hirsch length $h(G)$, because $h(G) = m+n$ and thus $\RF_G \preceq r^{4h(G)}$.
	
	Secondly, if $G \leq G'$ is a finite extension of $G$, meaning that $[G':G] = l < \infty$, then $G'$ is also linear but typically only embeds in $\GL(kl, \F)$ and not in $\GL(k,\F)$.  This results in the considerably weaker upper bound $r^{k^2l^2-1}$ compared to the new bound where the values $m$ and $n$ do not change under taking finite extensions. 
	
	We also communicate sharper bounds for certain subclasses in Section \ref{sec_upper}, more specifically in Theorem \ref{thm_upper}. In particular, the sharper bound for virtually abelian groups agrees with the exact result obtained in \cite{math_virt_ab}. In fact, we conjecture that the sharper, polylogarithmic bound on virtually nilpotent groups is also exact. Since the upper bound for torsion-free nilpotent groups only depends on its complex Mal'cev completion, this would positively answer \cite[Question 3]{math_virt_ab}.
	
Constructing matching lower bounds for these groups is usually a lot harder. In \cite{Mark_strict_dist}, the author showed that a finitely generated residually finite solvable group $G$ containing a cyclic exponentially distorted subgroup in its Fitting subgroup satisfies $r \preceq \RF_G$. Such a group is never virtually nilpotent, but not all $\mathcal{M}$-groups that are not virtually nilpotent have a cyclic exponentially distorted subgroup, as demonstrated by \cite[Example 7.1]{MR1736519}. In this paper, we prove that the bound $r\preceq \RF_G$ holds for all $\mathcal{M}$-groups that are not virtually nilpotent. 
	\begin{thmx} \label{thm_intro_lower}
		Let $G$ be a $\mathcal{M}$-group. If $G$ is virtually nilpotent, then $\RF_G \preceq \log^k$ for some $k\in\N$, otherwise $r \preceq \RF_G$.
	\end{thmx}
	\noindent In Section \ref{sec_lower}, we also communicate some sharper lower bounds and related open questions.

	The outline of this article is as follows. Section \ref{sec_prelim} introduces some background material, including the residual finiteness growth, nilpotent groups and Chebotarev's density theorem. In Section \ref{sec_minimax}, we introduce the class of solvable minimax groups, towards the characterization in Theorem \ref{thm_intro_minimax}. Sections \ref{sec_geom} and \ref{sec_upper} contain the proof of Theorem \ref{thm_intro_upper} and its refinements as given in Theorem \ref{thm_upper}, based on the notations introduced in Sectionr \ref{sec_nota}. The proof splits in two parts: first Section \ref{sec_geom} focuses on understanding the word norm $\norm{g}$ in $G$, and secondly Section \ref{sec_upper} constructs homomorphisms to finite groups $G\to Q$. Finally, in Section \ref{sec_lower}, we prove Theorem \ref{thm_intro_lower} and its refinements.
	\section{Preliminaries} \label{sec_prelim}
	This section consists of three parts. In subsection \ref{sec_RF}, we will recall the definition of residual finiteness growth. In subsection \ref{sec_lie_cor}, we recall the correspondence between nilpotent Lie groups and Lie algebras. This correspondence will play a central role in the proof of Theorem \ref{thm_intro_upper}. Subsection \ref{sec_Chebotarev} gives a brief introduction to Chebotarev's density theorem. We will use this theorem only once in the paper, namely in Proposition \ref{prop_delta_Zp}.
	\subsection{Residual Finiteness Growth} \label{sec_RF}
	In this subsection, we introduce the residual finiteness growth for residually finite groups, as it was originally introduced in \cite{bou2010quantifying}. 
	In the remainder, $G$ will be a group with neutral element $e \in G$ and the natural numbers $\N$ are equal to $\{1,2, \ldots\}$.
	
	Recall the following notions:
	\begin{df}
		A group $G$ is called \textbf{residually finite} if for every non-trivial element $g\in G$, there exists a homomorphism $\varphi: G \to Q$ to a finite group $Q$ such that $\varphi(g) \neq e$. 
	\end{df}
	\begin{df}
		Let $G$ be a finitely generated group, with finite generating set $S$. The \textbf{word norm} on $G$ via $S$ is defined as 
		$$\norm{g}_{G,S} = \min\{k \mid g = s_1^{\pm 1}\ldots s_k^{\pm 1}, s_i \in S, k\in \N\cup\{0\}\}.$$
	\end{df}
	\begin{nota}
		The word metric ball centered around $e$ with radius $r\in \R^+$, denoted by $B_{G,S}(r)$ is equal to
		$$B_{G,S}(r) = \{g\in G\mid \norm{g}_{G,S} \leq r\} = \{s_1^{\pm 1}\ldots s_k^{\pm 1}\mid s_i \in S, k\leq r\}.$$
		If $S$ is clear from the context, we will write $\norm{g}_G$ and $B_G(r)$.
	\end{nota}
	
	The residual finiteness growth was originally defined as a way to quantify the property `residual finiteness' using the word norm. It is defined as follows:
	\begin{df}
		The \textbf{divisibility function} $D_G: G\setminus\{e\} \to \N$ is defined as
		$$D_G(g) = \min\{[G:N] \mid g\notin N, N \lhd G\}.$$
	\end{df}
	Note that $D_G(g)$ is indeed well-defined for every $g\in G\setminus\{e\}$ by the definition of residual finiteness. Equivalently, $D_G(g)$ can be defined as the smallest size of $Q$ such that there exists a morphism $\varphi: G \to Q$ with $\varphi(g) \neq e$.
	\begin{df}
		The \textbf{residual finiteness growth} of $G$ with respect to $S$ is given by
		$$\RF_{G,S}: \R_{\geq 1} \to \N: r \mapsto \max\{D_G(g) \mid e\neq g \in B_{G,S}(r)\}.$$
	\end{df}
	This function, which a priori depends on the choice of $S$, becomes a group invariant if we consider this function up to the equivalence relation defined below.
	\begin{df}
		Let $f,g: \mathbb{R}_{\geq 1} \to \mathbb{R}_{\geq 1}$ be non-decreasing functions. We write 
		\begin{align*}f &\preceq g \Leftrightarrow \exists C >0: \forall r \geq \max\{1,1/C\}: f(r) \leq Cg(Cr);\\
			f&\approx g \Leftrightarrow f\preceq g \text{ and } g \preceq f.\end{align*}
	\end{df}
	Indeed, if $T$ is another choice of generating set, then there exists some $C > 0$ such that $B_{G,S}(r) \subset B_{G,T}(Cr)$, and hence $\RF_{G,S}(r) \leq \RF_{G,T}(Cr)$. Exchanging the roles of $S$ and $T$ shows that $\RF_{G,S} \approx \RF_{G,T}$. The same flexibility also allows us to replace the word norm $\norm{g}_{G,S}$ by norms that are not necessarily induced by a finite generating set, e.g.~the Euclidean norm on $\Z^m$.
	
	\subsection{Nilpotent Groups and Lie Algebras} \label{sec_lie_cor}
	In this subsection, we introduce nilpotent groups and their corresponding Lie algebras.
	\begin{df}
		A group $G$ is called \textbf{nilpotent} if there exists a central series, i.e.~a sequence of normal subgroups $G_i$ of $G$ such that
		$$\{e\} = G_{c+1} \lhd G_{c} \lhd \dots \lhd G_{2} \lhd G_{1} = G$$
		and $G_{i}/G_{i+1} \leq Z(G/G_{i+1})$. The minimal $c$ for which such a series exists is called the \textbf{nilpotency class} of the group G.
	\end{df}
	\begin{df}
		Let $G$ be a group. The $l$'th group of the lower central series, $\gamma_l(G)$, of $G$ is defined via the relation $\gamma_1(G) = G$ and $\gamma_{i+1}(G) = [\gamma_i(G),G]$.
	\end{df}
	
	Let $G$ be a finitely generated torsion-free nilpotent group. 
	As outlined in \cite[Section 4.2.2]{theory_nilpotent}, such a group has a central series
	\begin{equation} \label{eq_Malcev_basis_chain}
	\{e\}= G_{m+1} \lhd G_{h(G)} \lhd \ldots \lhd G_1  = G
	\end{equation}
	with $G_{i}/G_{i+1}$ infinite cyclic. This allows us to take elements $g_i\in G_i$ such that $G_i = \langle g_i, G_{i+1}\rangle$. 
	\begin{df}
		The set $\{g_1, \ldots , g_{m}\}$ as defined above is called a \textbf{Mal'cev basis} of $G$. The number $m$ is called the \textbf{Hirsch length} of $G$ and written as $h(G)$.
	\end{df}
	This basis satisfies the following properties (see \cite[Section 4.2.2]{theory_nilpotent}):
	\begin{prop} \label{prop_malcev_poly}
		Let $G$ be a finitely generated torsion-free nilpotent group. 
		Let $\{g_1, \ldots , g_m\}$ be a Mal'cev basis of $G$ corresponding to a central series as in Equation \eqref{eq_Malcev_basis_chain}. Then,
		\begin{enumerate}[(i)]
			\item every element $g\in G$ can be uniquely written as $g = g_1^{k_1}\cdots g_m^{k_m}$ with $k_i \in \Z$;
			\item $[G_i,G_j] \leq G_{\max(i,j)+1}$;
			\item there exist polynomials $f_i \in \Q[x_1, \ldots, x_{2m}]$ for $i=1, \ldots, m$ such that
			$$\left(g_1^{k_1}\cdots g_m^{k_m}\right) \cdot \left(g_1^{k'_1}\cdots g_m^{k'_m}\right) = g_1^{f_1(k_1, \ldots , k_m, k'_1, \ldots , k'_m)}\cdots g_m^{f_m(k_1, \ldots , k_m, k'_1, \ldots , k'_m)};$$
			\item there exist polynomials $f^\prime_i \in \Q[x_1, \ldots, x_m,z]$ for $i=1, \ldots, m$ such that
			$$\left(g_1^{k_1}\cdots g_m^{k_m}\right)^z = g_1^{f'_1(k_1, \ldots , k_m, z)}\cdots g_m^{f'_m(k_1, \ldots , k_m, z)}.$$
		\end{enumerate}
	\end{prop}
	\noindent This result allows us to identify a finitely generated torsion-free nilpotent group $G$ with the set $\Z^m$, where multiplication and taking exponents (including inversion) are defined by rational polynomials. It is part of \cite[section 4.3]{theory_nilpotent} that the following definition is well-defined. 
	\begin{df}
		Let $G$ be a finitely generated torsion-free nilpotent group and $R$ a ring containing $\Z[1/M]$, where $M$ is a common denominator of the polynomials $f_i$ and $f'_i$ ($1\leq i \leq m$) from Proposition \ref{prop_malcev_poly}. Denote $G^R$ for the $R$-\textbf{completion} of $G$, i.e. the group with $R^m$ as a set and its operations defined via the polynomials $f_i$ and $f'_i$. The $\Q$-completion $G^\Q$ of $G$ is also called the \textbf{rational Mal'cev completion}. 
	\end{df}
	Extending the last point in Proposition \ref{prop_malcev_poly} from $z\in \Z$ to $z\in \Q$, we see that $G^\Q$ is radicable:
	\begin{df}
		We say a torsion-free nilpotent group is \textbf{radicable} if for every $g\in G$ and $k \in \N$, there exists a unique $h\in G$ such that $h^k=g$.
	\end{df}
	\begin{nota}
		Let $k\in \N$ and $g\in G$, then the unique element $h\in G$ such that $h^k = g$ will be denoted by $g^{1/k}$. Similarly, we can define $g^q$ for every $q\in \Q$.
	\end{nota}
	Even if $G$ is not finitely generated one can construct the rational Mal'cev completion:
	\begin{thm} \label{thm_Malcev}
		Let $G$ be a torsion-free nilpotent group. There exists a torsion-free radicable nilpotent group $G^\Q$, called the rational Mal'cev completion, such that
		\begin{enumerate}[(i)]
			\item $G$ is a subgroup of $G^\Q$,
			\item for every $g\in G$, there exists $k \in \N$ such that $g^k \in G^\Q$.
		\end{enumerate}
		Moreover, the group $G^\Q$ is unique up to isomorphism.
	\end{thm}

The following notion of rank works for groups that are not necessarily finitely generated.
	\begin{df}
		Let $G$ be a group, then its \textbf{Prüfer rank} $r(G)\in \N\cup\{\infty\}$ is defined as the least value such that every finitely generated subgroup of $G$ can be generated by at most $r(G)$ elements.
	\end{df}
	
	We have the following well-known relation:
	\begin{lemma}
		Let $G$ be a torsion-free nilpotent group, then $r(G) = r(G^\Q)$. If $r(G)<\infty$, then $G$ contains a finitely generated subgroup $H$ such that $H^\Q = G^\Q$. 
	\end{lemma}
	\begin{proof}
	It is clear that if $r(G) = \infty$ then also $r(G^\Q) = \infty.$ So assume that $G$ is a torsion-free nilpotent group of finite Pr\"ufer rank. A refinement of the upper central series shows that $G$ is polyrational in the terminology of \cite[p.92]{robinson}. If $G$ has factors $G_i/G_{i+1} \leq \Q$ in its polyrational series, then $G^\Q$ has factors $G_i/G_{i+1}\otimes_\Z \Q = \Q$, so $r(G)$ and $r(G^\Q)$ are equal by \cite[Theorem 5.2.7]{robinson}. The finitely generated group $H$ is generated by elements $g_i \in G_i\setminus G_{i+1}$ for every $i$. 
	\end{proof}
	
	The Mal'cev correspondence in Theorem \ref{thm_Malcev_corr} below, see \cite[Theorem 10.11]{Khukhro}, gives a one-to-one correspondence between radicable nilpotent groups and nilpotent Lie $\Q$-algebras:
	\begin{df}
		Let $\mathfrak{g}$ be a \textbf{Lie algebra} with Lie bracket $[\cdot, \cdot]_L$ over a field $\F$. Define the lower central series $(\mathfrak{g}_i)_{i\in \N}$ of $\mathfrak{g}$ via $\mathfrak{g}_1 = \mathfrak{g}$ and $\mathfrak{g}_{i+1} = [\mathfrak{g}_i, \mathfrak{g}]_L$. A Lie algebra is \textbf{nilpotent} if there exists $c\in \N$ such that $\mathfrak{g}_{c+1} = 0$. The smallest such $c\in \N$ is called the \textbf{nilpotency class} of $\mathfrak{g}$.
	\end{df}

		Let $\mathfrak{g}$ be a nilpotent Lie algebra over a field $\F$. Define an operation on $\mathfrak{g}$ via the \textbf{Baker-Campbell-Hausdorff formula}:
		$$\ast : \mathfrak{g}\times\mathfrak{g} \to \mathfrak{g}: (v,w) \mapsto v\ast w := v+w + \frac{1}{2}[v,w]_L + \sum_{e=3}^\infty q_e(v,w), $$
		with $q_e(v,w)$ a specific rational linear combination of nested Lie bracket of length $e$, see for example \cite[Section 9.2]{Khukhro} for a more detailed description. Note that the Baker-Campbell-Hausdorff formula is defined as an infinite sum. However, since $\mathfrak{g}$ is nilpotent of, say, nilpotency class $c\in \N$, we know that $q_e(v,w) = 0$ for all $e>c$.

	\begin{thm}[\textbf{Mal'cev correspondence}] \label{thm_Malcev_corr}
		If $\mathfrak{g}$ is a nilpotent Lie algebra over $\Q$, then $(\mathfrak{g}, \ast)$ is a radicable nilpotent group. Furthermore, if $G$ is a radicable nilpotent group, then there exists a nilpotent Lie algebra $\mathfrak{g}$ over $\Q$ such that $G \cong (\mathfrak{g}, \ast)$ as groups.
	\end{thm}
	In fact, the Pr\"ufer rank of a radicable nilpotent group and the dimension of its corresponding Lie algebra are the same. Furthermore, under the isomorphism $G \cong (\mathfrak{g}, \ast)$, one can switch between multiplicative notation in $G$ and linear notation in $\mathfrak{g}$:
	\begin{prop} \label{prop_coor}
		Let $G$ be a finitely generated torsion-free nilpotent group with Mal'cev basis $\{g_1, \ldots , g_m\}$. If $G^\Q \cong (\mathfrak{g}, \ast)$, then the Mal'cev basis corresponds to a vector space basis of $\mathfrak{g}$.
		
		Furthermore, under this identification
		\begin{enumerate}[(i)]
			\item there exist rational polynomials $f_i \in \Q[x_1,\ldots,x_{2m}]$ for $i \in \{1, \ldots, m\}$ such that
			$$g_1^{z_1}\ast \ldots \ast g_m^{z_m} := \prod_{i=1}^m g_i^{z_i} = \sum_{i=1}^m f_i(z_1, \ldots , z_m)g_i;$$
			\item there exist rational polynomials $f^\prime_i \in \Q[x_1,\ldots,x_m]$ for $i \in \{1, \ldots, m\}$ such that
			$$\sum_{i=1}^m z_ig_i = \prod_{i=1}^m g_i^{f'_i(z_1, \ldots , z_m)} = g_1^{f'_1(z_1, \ldots , z_m)} \ast \ldots \ast g_m^{f'_m(z_1, \ldots , z_m)},$$
		\end{enumerate}
		holds for all $z_i \in \Q$.
	\end{prop}
	\begin{proof}
		The first part is proven in for example \cite[Theorem 6.7]{theory_nilpotent}. The second part follows from the Baker-Campbell-Hausdorff formula, see for example \cite[Lemma 4.4]{Assmann}.
	\end{proof}

	The Mal'cev correspondence also gives a relation between automorphisms of radicable groups and of their corresponding Lie algebras, as stated below. We have formulated the result for nilpotent group $\tilde{G}$ such that $G\leq \tilde{G} \leq G^\Q$, where $G$ is a torsion-free, finitely generated nilpotent group and $G^\Q$ is its $\Q$-completion. Note that the group $\tilde{G}$ can lie strictly between $G$ and $G^\Q$ in the sense that it does not need to be finitely generated nor radicable. In fact, most nilpotent groups under consideration will be of this type.
	
	\begin{prop} \label{prop_KQ_maps}
		Let $G$ be a finitely generated torsion-free nilpotent group with Mal'cev basis $\{g_1, \ldots, g_m\}$. Let $G \leq \tilde{G} \leq G^\Q$. Then, an automorphism $\varphi:\tilde{G} \to\tilde{G}$ extends uniquely to an automorphism $\varphi^\Q: G^\Q \to G^\Q$. Furthermore,
		\begin{enumerate}[(i)]
			\item under the identification $G^\Q \cong (\mathfrak{g}, \ast)$, group automorphisms of $G^\Q$ are Lie algebra automorphisms of $\mathfrak{g}$ and vice versa,
			\item a group automorphism of $G^\Q$ is given by a polynomial map with respect to the coordinates yielded by the Mal'cev basis.  
		\end{enumerate}
	\end{prop}
	\begin{proof}
		The first statement and (i) are given in \cite[Theorem 9.20]{Khukhro} and \cite[Theorem 10.13(f)]{Khukhro} respectively. We proceed to prove (ii). Take a general element $g = \prod_{i=1}^m g_i^{z_i}$ with $z_i \in \Q$ of $G^\Q$. By the first part of Proposition \ref{prop_coor}, this equals
		\begin{equation} \label{eq_g_as_sum}
			g = \sum_{i=1}^m f_i(z_1, \ldots , z_m)g_i.
		\end{equation}
		Now, applying the automorphism $\varphi$ to $g$ means applying a linear (and therefore also polynomial) map on this expression by statement (i). Now, use the second part of Proposition \ref{prop_coor} to rewrite the expression into a product form. Since the composition of polynomials is still a polynomial, we conclude that $\varphi(g) = \prod_{i=1}^m g_i^{z_i'}$, where every $z_i'$ is a rational polynomial in $\{z_1, \ldots , z_m\}$.
	\end{proof}
	
	In this paper, we will also work with a notion that is slightly weaker than being a Lie algebra: 
	\begin{df}
		Let $R$ be a (commutative) ring. We call $(L, [\cdot, \cdot]_L)$ a \textbf{Lie ring} if it is an algebra over the ring $R$ satisfying $[v,v]_L = 0$ for all $v\in L$ and satisfying the Jacobi identity.
	\end{df}
	In particular, we will work with a finitely generated torsion-free group $G$ such that the following holds: under the identification of $G^\Q \cong (\mathfrak{g}, \ast)$ the group $G$ is not only a subgroup of $G^\Q$ but also a Lie ring over $\Z$ inside $\mathfrak{g}$. In \cite{sega83-1}, they call such a group an LR-group (short for Lie ring group).

	\subsection{Chebotarev's density theorem} \label{sec_Chebotarev}
	In this section, we will briefly introduce the reader to Chebotarev's density theorem, which is a classical result from Number Theory. We will only use this result once in this paper: we will apply Proposition \ref{prop_apply_cheby} in the proof of Proposition \ref{prop_delta_Zp}. More details about the results in this subsection can be found in several standard works, e.g. \cite[Chapters 2-4]{marcus1977number}.
	
	\begin{nota}
		Let $\F$ be a \textbf{number field}, i.e. a finite field extension of $\Q$. We will suppose that $\F$ is Galois over $\Q$ with Galois group $\Gal(\F/\Q) = \{\sigma_1, \ldots , \sigma_n\}$.
	\end{nota}
		Recall that an \textbf{algebraic integer} is a zero of a univariate polynomial over $\Z$. The ring of algebraic integers in a number field $\F$ will be denoted by $\mathcal{O}_\F$. 

	\begin{ex}
		The algebraic integers in $\Q$ are precisely the integers: $\mathcal{O}_\Q = \Z$. 
	\end{ex}
	
	Chebotarev's density theorem treats the relationship between prime ideals in $\Z$, i.e. $p\Z$ for prime numbers $p$, and prime ideals in $\mathcal{O}_\F$. In general, if $x\in \mathcal{O}_\F$, then $x\mathcal{O}_\F$ is an ideal. Since $\mathcal{O}_\F$ is a Dedekind domain by \cite[Theorem 14]{marcus1977number}, there exists a decomposition in (not necessarily distinct) prime ideals of the form
	$$x\mathcal{O}_\F = \mathfrak{P}_1\mathfrak{P}_2\dots\mathfrak{P}_l$$
	for some $l\in \N$. This decomposition is unique up to ordering. Specifically for $p\mathcal{O}_\F$ with $p\in\N$ prime, we observe the following (see \cite[Chapter 2 \& 3]{marcus1977number}): 
	\begin{prop}
		Let $\F$ be a number field, Galois over $\Q$, with Galois group $G = \Gal(\F/\Q)$ and $p$ be a prime number. The following holds:
		\begin{enumerate}[(i)]
			\item There is a unique decomposition of the ideal $p\mathcal{O}_\F \subset \mathcal{O}_\F$ into prime ideals (up to ordering), i.e.
			$$p\mathcal{O}_\F = \mathfrak{P}_1^{e}\mathfrak{P}_2^{e}\dots\mathfrak{P}_l^{e},$$
			where $l, e\in\mathbb{N}$.
			\item All quotients $\F_{\mathfrak{P}_i} := \mathcal{O}_\F/\mathfrak{P}_i$ are isomorphic finite fields of characteristic $p$.
			\item The group $G$ induces a transitive action on $P = \{\mathfrak{P}_1, \dots \mathfrak{P}_l\}$ via $\sigma \cdot \mathfrak{P} = \sigma(\mathfrak{P})$.
		\end{enumerate}
	\end{prop}
	The stabilizer of $\mathfrak{P}_j \in \{\mathfrak{P}_1, \dots \mathfrak{P}_l\}$ with respect to the transitive action above is called the decomposition group of $\mathfrak{P}_j$ over $p\Z$:
	\begin{df}
		Let $p\mathcal{O}_\F = \mathfrak{P}_1^{e}\mathfrak{P}_2^{e}\dots\mathfrak{P}_l^{e}$. For $1\leq j \leq l$, define the \textbf{decomposition group} of $\mathfrak{P}_j$ over $p\Z$ by
		$$D(\mathfrak{P}_j\mid p\Z) = \{\sigma \in G\mid \sigma(\mathfrak{P}_j) = \mathfrak{P}_j\} \subset G.$$
	\end{df}
	Fix $\mathfrak{P}_j \in \{\mathfrak{P}_1, \dots \mathfrak{P}_l\}$, and let $f = [\F_{\mathfrak{P}_j}: \Z_p]$. Since $|G| = ref$ by \cite[Theorem 21]{marcus1977number} and $G$ acts transitively on $\{\mathfrak{P}_1, \dots \mathfrak{P}_l\}$, one sees that the stabilizer of $\mathfrak{P}_j$, which is $D(\mathfrak{P}_j\mid p\Z)$ by definition, has $ef$ elements. In the particular case that $e = 1$, we say that $p\Z$ is \textbf{unramified} in $\mathcal{O}_\F$. Then, we have
	$$|D(\mathfrak{P}_j\mid p\Z)| = f .$$
	Now, we will define the Frobenius and Artin symbol for the unramified primes $p\Z$.
	
	\begin{lemma}
		If $p\Z$ is unramified in $\mathcal{O}_\F$ with $p\mathcal{O}_\F = \mathfrak{P}_1\mathfrak{P}_2\dots\mathfrak{P}_l$, then there is an isomorphism
		$$\Psi_\mathfrak{P} : D(\mathfrak{P}\mid p\Z) \to \Gal(\mathbb{F}_\mathfrak{P}/\Z_p)$$
		for every $\mathfrak{P} \in \{\mathfrak{P}_1, \dots \mathfrak{P}_l\}$.
	\end{lemma} 
	\begin{proof}
		Take any $\sigma \in D(\mathfrak{P}\mid p\Z)$. Since $\sigma(\mathfrak{P}) = \mathfrak{P}$, the isomorphism $\sigma: \mathcal{O}_\F \to \mathcal{O}_\F$ induces an isomorphism $\bar{\sigma}: \mathbb{F}_\mathfrak{P} \to \mathbb{F}_\mathfrak{P}$. The map $\Psi_\mathfrak{P}$ is then defined as $\Psi_\mathfrak{P}(\sigma) = \bar{\sigma}$, with the remainder of the lemma given in \cite[p. 71]{marcus1977number}. 
	\end{proof}
	\begin{df}
		The \textbf{Frobenius symbol} $\left[\dfrac{\mathbb{F}/\mathbb{Q}}{\mathfrak{P}}\right]$ is the group element of $D(\mathfrak{P}\mid p\Z)$ given by $\Psi_\mathfrak{P}^{-1}(\tilde{\sigma})$, where $\tilde{\sigma}:\mathbb{F}_\mathfrak{P}\to \mathbb{F}_\mathfrak{P}: x \mapsto x^{p^f}$ is the Frobenius automorphism (which is a generator of $\Gal(\mathbb{F}_\mathfrak{P}/\Z_p)$). The \textbf{Artin symbol} $\left(\dfrac{\mathbb{F}/\mathbb{Q}}{p\Z}\right)$ is the conjugacy class of $\left[\dfrac{\mathbb{F}/\mathbb{Q}}{p\Z}\right]$ in $G$.
	\end{df}
	\begin{rem}
		Note that the Frobenius symbol hence represents a generator of $D(\mathfrak{P}\mid p\Z)$ and the Artin symbol its conjugacy class. The Artin symbol is independent of the choice of $\mathfrak{P} \in \{\mathfrak{P}_1, \dots \mathfrak{P}_l\}$, as
		$$\left[\dfrac{\mathbb{F}/\mathbb{Q}}{\sigma(\mathfrak{P})}\right] = \sigma\left[\dfrac{\mathbb{F}/\mathbb{Q}}{\mathfrak{P}}\right]\sigma^{-1}.$$
	\end{rem}
	
	We are now ready to state the theorem. (Our formulation is a weakened version of the one given in \cite[Theorem 3.2]{Serre}.)
	\begin{thm}[\textbf{Chebotarev's density theorem}]
		Using the notation introduced in this section, let $\mathcal{C}$ denote a conjugacy class in $G$, and let $\pi_\mathcal{C}(r)$ denote the number of prime numbers $p\leq r$ such that
		$$p\Z \text{ is unramified in } \mathcal{O}_\mathbb{F}\text{ and } \left(\dfrac{\mathbb{F}/\mathbb{Q}}{p\Z}\right) = \mathcal{C}.$$
		Then, $\pi_\mathcal{C}(r) \asymp r/\log(r)$, i.e. there exist constants $C_1,C_2 > 0$ such that for all $r$ sufficiently large
		$$C_1r/\log(r) \leq \pi_\mathcal{C}(r) \leq C_2r/\log(r).$$
	\end{thm}
	We will be interested in the case where $\mathcal{C} = \{\Id\}$. This gives the following formulation:
	\begin{prop} \label{prop_apply_cheby}
		Using the notation introduced in this section, let $\pi_{\mathcal{O}_\F \to \Z_p}(r)$ denote the number of prime numbers $p\leq r$ such that there exists a homomorphism $\rho: \mathcal{O}_\F \to \Z_p$. Then, $\pi_{\mathcal{O}_\F \to \Z_p}(r) \asymp r/\log(r)$, i.e. there exist constants $C_1,C_2 > 0$ such that for all $r$ sufficiently large
		$$C_1r/\log(r) \leq \pi_{\mathcal{O}_\F \to \Z_p}(r) \leq C_2r/\log(r).$$
	\end{prop}
	\begin{proof}
		Apply Chebotarev's density theorem to the conjugacy class $\mathcal{C} = \{\Id\}$, or thus with $D(\mathfrak{P}\mid p\Z)$ trivial. In particular, $1 = |D(\mathfrak{P}\mid p\Z)| = f$, since $p\Z$ is unramified in $\mathcal{O}_\F$, so with $\F_\mathfrak{P} = \Z_p$. Hence we obtain a homomorphism $\rho: \mathcal{O}_\F \to \F_\mathfrak{P} = \Z_p$. 
	\end{proof}
	It is this result that we will apply in Proposition \ref{prop_delta_Zp}. Note that this result remains valid if we exclude a finite amount of primes $p$ from the statement. In particular, given $x\in \mathcal{O}_\F$, we may suppose that $\rho(x) \neq 0$. Indeed, $x\mathcal{O}_\F$ decomposes as a product of finitely many prime ideals in $\mathcal{O}_\F$,
	$$x\mathcal{O}_\F = \mathfrak{P}_1\mathfrak{P}_2\dots\mathfrak{P}_l$$
	for some $l\in \N$. Now, $\mathfrak{P}_i\cap \Z = p_i\Z$ for some prime $p_i$. Excluding the primes $\{p_i\mid 1\leq i \leq l\}$ from the proposition above then guarantees that $\rho(x) \neq 0$.

	\section{Minimax Groups} \label{sec_minimax}
	In this section, we will prove Theorem \ref{thm_intro_minimax} (see Theorem \ref{thm_charac_minimax} and Proposition \ref{prop_corr_M_groups}), which provides a characterization of finitely generated residually finite virtually solvable minimax groups via a short exact sequence of the form
	$$1\to K \to G \to H \to 1,$$
	with $K$ nilpotent and $H$ virtually abelian. Apart from this characterization of $\mathcal{M}$-groups as we define them, no other results or definitions of this section will be used in the rest of the article. 

	\begin{df}
		We say a solvable group $G$ is \textbf{minimax} if there exists a series $1= G_0 \lhd G_1 \lhd \ldots \lhd G_n = G$, such that each factor $G_{i+1}/G_i$ satisfies max or min.
	\end{df}
	Recall that a group satisfies max if every increasing series of subgroups stabilizes after a finite number of steps, and analogously for min for every decreasing series of subgroups.
	\begin{df} \label{df_Mgroup}
		We say a finitely generated group $G$ is an \textbf{$\mathcal{M}$-group} if there exists a short exact sequence of the form
		$$1 \to K \to G \to H \to 1,$$
		where $K$ is a torsion-free nilpotent group of finite Pr\"ufer rank and $H$ is virtually abelian.
	\end{df}
	The goal of this section is to give proof of the following statement:
	\begin{thm} \label{thm_charac_minimax}
		Let $G$ be a finitely generated group, then $G$ is a residually finite virtually solvable minimax group if and only if $G$ is an $\mathcal{M}$-group.
	\end{thm}

For the first implication, let $G$ be a finite extension of a residually finite solvable minimax group $\tilde{G}$. By \cite[Theorem 5.2.2]{robinson}, we have:
	\begin{thm} \label{thm_Fit_in_minimax}
		Let $G$ be a solvable minimax group. Then, the Fitting subgroup $\Fit(G)$ is nilpotent and $G/\Fit(G)$ is virtually abelian.
	\end{thm}
	Hence, applying this to $\tilde{G}$, we obtain a short exact sequence of the form
	$$1 \to \Fit(\tilde{G}) \to \tilde{G} \to \tilde{G}/\Fit(\tilde{G}) \to 1.$$
	Since $\Fit(\tilde{G})$ is characteristic in $\tilde{G}$, it is normal in $G$. This yields a short exact sequence
	$$1\to \Fit(\tilde{G}) \to G \to H \to 1.$$
	By construction, we have the following information:
	\begin{itemize}
		\item The group $H$ is an extension of $\tilde{G}/\Fit(\tilde{G})$, which is a finitely generated virtually abelian group, by the finite group $G/\tilde{G}$. Hence, $H$ is virtually abelian itself.
		\item The group $\Fit(\tilde{G})$ is nilpotent. As a subgroup of $\tilde{G}$, it is also residually finite and solvable minimax.
	\end{itemize}
	Now, we will show that we can replace $\Fit(\tilde{G})$ by a torsion-free nilpotent group.
	\begin{prop} \label{prop_Fit_to_K}
		Let $G$ be a residually finite, finitely generated group with a normal subgroup $N$ that is nilpotent and minimax. Then, there exists a torsion-free subgroup $K\leq N$ such that $K\lhd G$ and $[N:K] < \infty$.
	\end{prop}
	\begin{proof}
		Recall that the torsion elements of $N$ form a fully invariant subgroup $T$ of $N$, since $N$ is nilpotent, see \cite[Lemma 1.2.13]{robinson}. We will first argue that $T$ is finite. For this, note that $T$ can also be described as the unique largest normal torsion subgroup $\tau(N)$ of $N$. By the remark below \cite[Proposition 5.2.1]{robinson}, the group $\tau(N)$ is Cernikov, i.e.~it is virtually a direct product of finitely many quasicyclic groups. However, $N$ is residually finite and thus can have no quasicyclic subgroups by \cite[Corollary 5.3.2]{robinson}. Hence, $T = \tau(N)$ has to be virtually trivial, thus finite.
		
		Take $e\neq t\in T$ arbitrary. Since $G$ is residually finite, we can find a homomorphism to a finite group $\varphi_t: G \to Q_t$ such that $\varphi_t(t) \neq e$. Now, 
		$$K = N\cap \left(\bigcap_{t\in T}\ker\varphi_t\right) $$
		is a normal subgroup of $G$ such that $K \subset N\setminus T$. Hence, it is torsion-free nilpotent. Its index in $N$ is finite, since $\cap_{t\in T}\ker\varphi_t$ has finite index in $G$.
	\end{proof}
	Let $N = \Fit(\tilde{G})$ in the result above. The normal subgroup $K$ yields a short exact sequence of the form
	$$1\to K \to G \to \tilde{H} \to 1.$$
	Now, the group $\tilde{H}$ is an extension of the finite $\Fit(\tilde{G})/K$ by the finitely generated, virtually abelian group $H$. This implies that $\tilde{H}$ is virtually abelian. 
	Since $G$ is solvable minimax, it has finite Pr\"ufer rank (see e.g. \cite[Lemma 5.1.6]{robinson}). Hence, its subgroup $K$ has finite Pr\"ufer rank too. This finishes the proof of Theorem \ref{thm_charac_minimax}.

Next, we show the other implication of Theorem \ref{thm_charac_minimax}. 
	\begin{prop} \label{prop_corr_M_groups}
		An $\mathcal{M}$-group is a finite extension of a residually finite, finitely generated (torsion-free) solvable minimax group.
	\end{prop}
	\begin{proof}
		Suppose that $G$ is an $\mathcal{M}$-group with corresponding short exact sequence
		$$1 \to K \to G \to H \to 1,$$
		as in Definition \ref{df_Mgroup}. Since $H$ is finitely generated and virtually abelian, it contains a free abelian subgroup $\Z^n$ which has finite index in $H$. The preimage of this subgroup in $G$, denoted by $\bar{G}$, is a finite index normal subgroup of $G$ with short exact sequence
		$$1 \to K \to \bar{G} \to \Z^n \to 1.$$
		We will argue that $\bar{G}$ is a residually finite, finitely generated torsion-free, solvable minimax group.
		
		By definition, $G$ is finitely generated. Since $\bar{G} \lhd_f G$, $\bar{G}$ is finitely generated too. It is clear that $\bar{G}$ is torsion-free, since $K$ and $\Z^n$ are. For solvability, note that $[\bar{G}, \bar{G}] \leq K$ and $K$ is nilpotent (and therefore solvable), hence $\bar{G}$ is solvable. The group $\bar{G}$ has finite Pr\"ufer rank, as both $K$ and $\Z^n$ have finite Pr\"ufer rank. However, by \cite[Corollary 10.5.3]{robinson}, a finitely generated solvable group with finite Prüfer rank is minimax, so $\bar{G}$ is minimax.	Finally, by \cite[Theorem 5.1.8]{robinson}, a torsion-free, solvable minimax group is linear, and thus $\bar{G}$ is linear. We conclude by noting that finitely generated, linear groups are residually finite.
	\end{proof}

	We end this section with some examples of $\mathcal{M}$-groups.
	\begin{ex} 
		A polycyclic group is a residually finite, finitely generated solvable minimax group, as these groups are max. Hence, a virtually polycyclic group is an $\mathcal{M}$-group. 
	\end{ex}
	\begin{ex} \label{ex_BS}
		It is a known fact that the \textbf{Baumslag-Solitar groups} $\BS(1,n)$ with $0\neq n \in \Z$, defined via the presentation $\langle x,y \mid y^{-1}xy = x^n \rangle$, are isomorphic to $\Z[1/n] \rtimes_\varphi \Z$ where \begin{align*}
	\varphi(1): \Z[1/n] &\to \Z[1/n] \\
	 x &\mapsto n x.
		\end{align*} 
		These $\mathcal{M}$-groups show that the torsion-free nilpotent subgroup $K = \Z[1/n]$ is not necessarily finitely generated.
	\end{ex}
	\section{Setup and Notations} \label{sec_nota}
The notations introduced in this section will be used throughout sections \ref{sec_geom} and \ref{sec_upper}. The finitely generated groups $G$ in this section fit in a short exact sequence of the form
	$$1\to K \to G \to H \to 1,$$
	with $K$ a torsion-free nilpotent group of finite Pr\"ufer rank and $H$ finitely generated virtually abelian. 
	
	For the ease of referencing, we will first state our notations and conventions and then only afterwards prove that the notations make sense. For example, we will introduce another related group $\bar{K}$, and we will show below that this group exists with the given properties.
	\begin{nota} \label{nota_group}
		We fix the groups $G$, $\bar{G}$, $K$, $\bar{K}$, $H$ and $\Z^n$ as follows:
		\begin{itemize}
			\item The group $G$ will be a fixed $\mathcal{M}$-group that fits in the short exact sequence 
			$$ 1 \to K \to G \to H \to 1.$$
			\item The group $K$ is torsion-free nilpotent, and $r(K^\Q) = m$. This implies that there exists a finitely generated subgroup $\bar{K} \leq K$ such that $\bar{K}^\Q = K^\Q$. 
			\item The group $H$ is finitely generated virtually abelian with a free abelian subgroup $\Z^n$ of finite index, i.e. $\Z^n \lhd_f H$.
			\item The group $\bar{G} \lhd_f G$ is the preimage of $\Z^n$ in $G$. In particular, it fits in the short exact sequence
			$$1 \to K \to \bar{G} \to \Z^n \to 1.$$
		\end{itemize}
	\end{nota}
	\begin{nota} \label{nota_basis}
	There exists a finitely generated $\bar{K}$ as above and a generating set $\{k_i, h_j, f_s \mid 1\leq i \leq m, 1\leq j \leq n, 1 \leq s \leq [H: \Z^n]-1\}$ of the group $G$ as follows:
		\begin{itemize}
			\item The subset $\{k_i\mid 1 \leq i \leq m\}$ is a Mal'cev basis of $\bar{K}$.
			\item The set $\{h_jK \in H \cong G/K \mid 1 \leq j \leq n\}$ give the standard generators of $\Z^n$ in $H$. Furthermore, the set $\{k_i, h_j \mid 1\leq i \leq m, 1\leq j \leq n\}$ generates $\bar{G}$. In the special case when $\bar{G} = K \rtimes_\varphi \Z^n$ for $\varphi: \Z^n \to \Aut(K)$, the elements are equal to $h_j = (e,e_j)$, where $e_j$ is the $j$-th standard vector. 
			\item The set $\{f_s\bar{G} \in (G/\bar{G}) \cong (H/\Z^n)\mid 1\leq s \leq [H: \Z^n]-1\}$ are precisely all non-trivial elements of $H/\Z^n$.
		\end{itemize}
	\end{nota}
	We now show that $\bar{K}$ and the generating set indeed exist.
	\begin{lemma} \label{lem_basis_exists}
		There exists a finitely generated subgroup $\bar{K}$ of $K$ with $\bar{K}^\Q = K^\Q$ such that there exists a generating set of $G$ as defined in Notation \ref{nota_basis}.
	\end{lemma}
	\begin{proof}
		The group $\bar{G}$ is finitely generated as a finite index subgroup of the group $G$. Choose $h_1$ to $h_n$ to be preimages of the standard basis of $\Z^n$. If $\bar{G} = K\rtimes_\varphi H$, we take the obvious preimages $(e,e_j)$, where $e_j$ is the $j$'th standard vector of $\Z^n$. Let $\pi: \bar{G} \to \Z^n$ denote the projection of $\bar{G}$ onto $\bar{G}/K \cong \Z^n$. As the elements $\pi(h_1), \ldots, \pi(h_n)$ generate $\pi(\bar{G})$ and $\bar{G}$ is finitely generated, there exists a finite set $S \subset K$ such that $S \cup \{h_1, \ldots, h_n\}$ still generates $\bar{G}$. 
		
		Since $r(K) < \infty$, we can take a finite set $\tilde{S} \subset K$ such that $\langle \tilde{S} \rangle$ has $K^\Q$ as its Mal'cev completion.
		Define $\bar{K}$ to be the group generated by $S \cup \tilde{S}$, and take a Mal'cev basis $\{k_1, \ldots, k_m\}$ of $\bar{K}$. Note that $\{k_i, h_j \mid 1\leq i \leq m, 1 \leq j \leq n\}$ now generates $\bar{G}$, because $S \subset \langle k_1, \ldots k_m \rangle$. 
		To finish the proof, it now suffices to take the $f_s$ to be preimages in $G$ of the non-trivial elements of $G/\bar{G}$.
	\end{proof}
	\begin{nota} \label{nota_action}
		Note that conjugation by the generators $\{h_j, f_s\mid 1\leq j \leq n, 1 \leq s \leq [H: \Z^n]-1\}$ induces automorphisms on $K$, since $K$ is a normal subgroup of $G$. We will also fix this notation:
		\begin{itemize}
			\item Let $\xi_j: K \to K$ denote the isomorphism such that $h_jk = \xi_j(k)h_j$ for all $k\in K$.
			\item Let $\eta_s: K \to K$ denote the isomorphism such that $f_sk = \eta_s(k)f_s$ for all $k\in K$.
		\end{itemize}
	\end{nota}
	\begin{nota} \label{nota_poly}
		Finally, we also fix some constants. Here, all mentioned polynomials are considered with respect to the Mal'cev basis $\{k_i\mid 1\leq i \leq m\}$ of $\bar{K} \leq K^\Q$. Recall that this is also a vector space basis under the identification of $K^\Q$ with its Lie algebra by Proposition \ref{prop_coor}.
		\begin{itemize}
			\item Consider the rational polynomials defining multiplication and exponentiation on $K^\Q$  as in Proposition \ref{prop_malcev_poly}. Let $\Delta_K \in \N$ denote a common denominator of these polynomials.
			\item Consider the rational polynomials defining the automorphisms $$\{\xi_j, \xi_j^{-1}, \eta_s, \eta_s^{-1}\mid 1\leq j \leq n, 1\leq s \leq [H: \Z^n]-1\}$$ as in Proposition \ref{prop_KQ_maps}. Let $\Delta_{\text{Hom}} \in \N$ denote a common denominator.
			\item Let $\epsilon_i, \epsilon_j \in \{1,-1\}$. The element $[h_i^{\epsilon_i}, h_j^{\epsilon_j}]$ must lie in $K\leq K^\Q$, since $[h_i^{\epsilon_i}, h_j^{\epsilon_j}]K = K \in \bar{G}/K \cong \Z^n$. Hence, it can be written as $k_1^{z_1}\dots k_m^{z_m}$ with $z_i \in \Q$. Let $\Delta_H$ denote a common denominator of all these rational coordinates, for all possible $[h_i^{\epsilon_i}, h_j^{\epsilon_j}]$ with $1 \leq i,j \leq n$.
			\item Under the identification $K^\Q \cong (\mathfrak{k}, \ast)$, let $\Delta_{\text{BCH}}$ denote a common denominator of the  polynomials in Proposition \ref{prop_coor} and of the rational coefficients of the Baker-Campbell-Hausdorff formula for $K$.
			\item Let $\Delta$ denote the product of all those constants. In particular, $\Delta$ is a common denominator for all the numbers mentioned above.
		\end{itemize}
	\end{nota}
	The choice of $\Delta$ in Notation \ref{nota_poly} is chosen such that the following result holds:
	\begin{lemma} \label{lem_multiplications}
		Take notations as in \ref{nota_poly}, then the following statements hold:
		\begin{itemize}
			\item If $k_1^{z_1}\dots k_m^{z_m}$ and $k_1^{z'_1}\dots k_m^{z'_m}$ satisfy $z_i, z_i' \in \Z[1/\Delta]$ for all $1\leq i \leq m$, then their product 
			$$k_1^{z_1}\dots k_m^{z_m} \cdot k_1^{z'_1}\dots k_m^{z'_m} = k_1^{z''_1}\dots k_m^{z''_m} $$
			also satisfies $z_i'' \in \Z[1/\Delta]$. The same conclusion holds for exponentiation and inversion.
			\item A product of the form $h_j^{\pm 1}\cdot k_1^{z_1}\dots k_m^{z_m}$ with all $z_i \in \Z[1/\Delta]$ equals $k_1^{z'_1}\dots k_m^{z'_m}h_j^{\pm 1}$ with all $z'_i \in \Z[1/\Delta]$.
			\item A product of the form $h_i^{\epsilon_i}h_j^{\epsilon_j}$ with $\epsilon_i, \epsilon_j \in \{1,-1\}$ equals $k_1^{z_1}\dots k_m^{z_m}h_j^{\epsilon_j}h_i^{\epsilon_i}$ with all $z_k \in \Z[1/\Delta]$.
		\end{itemize}
	\end{lemma}
	\begin{proof}
		 These three observations follow from the choice of $\Delta_K$, $\Delta_\text{Hom}$ and $\Delta_H$ respectively.
	\end{proof}
	A direct consequence of this result is the following observation, which we will use throughout the next sections:
\begin{lemma} \label{lem_g_as_product}
Take notations as in Notation \ref{nota_group} and \ref{nota_basis}, then every element $g\in G$ can be uniquely written as 
	$$g = k_1^{z_1}\dots k_m^{z_m}h_1^{l_1}\dots h_n^{l_n}f',$$
	with $z_i \in \Z[1/\Delta]$, $l_j \in \Z$ and $f' \in \{e, f_s \mid 1\leq s \leq [H: \Z^n]-1\}$. Moreover, if $g\in K$, then $f' = e$ and $l_j = 0$ for all $1\leq j \leq n$. 
\end{lemma}
\begin{proof}
	We first show that every element $\bar{g}$ in $\bar{G}$ can be uniquely written as $\bar{g} = k_1^{z_1}\dots k_m^{z_m}h_1^{l_1}\dots h_n^{l_n}$ with $z_i \in \Z[1/\Delta]$ and $l_j \in \Z$. Take $e \neq \bar{g} \in \bar{G}$ arbitrarily. We must first show that we can write it in the given form. Since $\bar{G} = \langle k_i, h_j \mid 1\leq i\leq m, 1\leq j \leq n\rangle$, we know that 
	$$\bar{g} = \prod_{k=1}^d g_k$$ 
	with $d\in \N$ and $g_k \in \{k_i, k_i^{-1}, h_j, h_j^{-1}\mid 1\leq i \leq m, 1\leq j\leq n\}$. Using induction on $d \in \N$ and Lemma \ref{lem_multiplications}, the existence follows easily. To argue that this expression is unique in $\bar{G}$, suppose that
	$$ k_1^{z_1}\dots k_m^{z_m}h_1^{l_1}\dots h_n^{l_n} = k_1^{z'_1}\dots k_m^{z'_m}h_1^{l'_1}\dots h_n^{l'_n}.$$ If we project this equality onto $\Z^n \cong \bar{G}/K$, then we see that $l_j = l_j'$ for all $1\leq j\leq n$, since $\{h_j\mid 1\leq j \leq n\}$ projects onto a basis of $\Z^n$. Now, it suffices to show that $k_1^{z_1}\dots k_m^{z_m} = k_1^{z'_1}\dots k_m^{z'_m}$ implies that $z_i = z_i'$, but this follows from the uniqueness of this expression in $K^\Q$. As $f^\prime$ is determined by the projection to $H/\Z^n \approx G/\bar{G}$, the last part of the statement follows.	
\end{proof}
	
	We will also use the Mal'cev correspondence as in Theorem \ref{thm_Malcev_corr} to fix in $\bar{K}$ some subset $L$ which is both a subgroup and a Lie ring. We will use the following notation:
		\begin{nota} \label{nota_Lie}
		We will identify $K^\Q$ with its corresponding Lie algebra $\mathfrak{k}$, coming from the Mal'cev correspondence as in Theorem \ref{thm_Malcev_corr}. As a consequence, the extensions of the maps $\xi_j$ and $\eta_k$ to $K^\Q$, which we denote with the same symbol, are both group and algebra homomorphisms. Note that $K$ is now seen as a subset of a Lie algebra. We say a Lie ring/algebra $\mathfrak{g} \leq \mathfrak{k}$ is $\mathcal{H}$-invariant if $\xi_j(\mathfrak{g}) = \mathfrak{g}$ and $\eta_s(\mathfrak{g}) = \mathfrak{g}$ for all $1\leq j \leq m$ and $1\leq s \leq [H: \Z^n]-1$. If we write $\xi \in \mathcal{H}$, then we mean $\xi \in \{\xi_j^{\pm 1}, \eta_s^{\pm 1}\mid 1\leq j \leq m, 1\leq s \leq [H:\Z^n]-1\}$.
	\end{nota}
	\begin{nota} \label{nota_LR_groups}
		We can fix a Lie ring $L$ and a number $\Delta$ as a multiple of the one in Notation \ref{nota_poly} such that
		\begin{itemize}
			\item $L \subset \bar{K}$,
			\item $(L, \ast)$ is a group,
			\item $K \subset L\otimes_\Z \Z[1/\Delta]$,
			\item $L\otimes_\Z \Z[1/\Delta]$ is $\mathcal{H}$-invariant, and
			\item $(L\otimes_\Z \Z[1/\Delta], \ast)$ is a group.
		\end{itemize}
		If $R$ is a ring, we will denote $L\otimes_\Z R$ by $L^R$. The notation $L^{\Z[1/\Delta]}$ will be shortened to $L^\Delta$. (See Lemma \ref{lem_L_exists} for the existence of $L$ and $\Delta$.)
	\end{nota}
	\begin{lemma} \label{lem_L_exists}
		Given the notation introduced in Notations \ref{nota_group}-\ref{nota_Lie}, there exists a Lie ring $L$ and a constant $\Delta_L\in \N$ such that
		$$L \subset \bar{K} \subset K \subset L\otimes_\Z \Z[1/(\Delta_L\Delta)].$$
		Furthermore, we may suppose that $L\otimes_\Z \Z[1/(\Delta_L\Delta)]$ is $\mathcal{H}$-invariant, and both Lie rings are also groups for $\ast$ determined by Baker-Campbell-Hausdorff.
	\end{lemma}
	\begin{proof}
		The existence of a Lie-ring $L$ such that $(L, \ast)$ is a group and $L \lhd_f K$ is guaranteed by \cite[Chapter 6, part B]{sega83-1}. Note in particular, that $L^\Q := \Q L = \mathfrak{k}$. Take a basis of $L$. Let $\Delta_L$ denote a common denominator of the (rational) entries of both  the matrices representing the base changes between the basis on $L$ and the Mal'cev basis on $\bar{K}$, and the matrices representing $\xi_j^{\pm 1}$ and $\eta_s^{\pm 1}$ with respect to the basis on $L$. 
		
		Now take any $g\in K$. By Lemma \ref{lem_g_as_product}, we know that $g = k_1^{z_1}\dots k_m^{z_m}$ with $z_i \in \Z[1/ \Delta]$. By the choice of $\Delta_{BCH}$ in Notation \ref{nota_poly}, we observe that $g = \lambda_1k_1 + \ldots + \lambda_mk_m$ for $\lambda_i \in \Z[1/ \Delta]$. Hence, with respect to the basis on $L$, it surely has coordinates in $\Z[1/(\Delta_L\Delta)]$. Therefore, $K \subset L\otimes_\Z \Z[1/(\Delta_L\Delta)]$. By the choice of $\Delta_L$ we also immediately conclude that $ L\otimes_\Z \Z[1/(\Delta_L\Delta)]$ must be $\mathcal{H}$-invariant. The fact that $(L\otimes_\Z\Z[1/(\Delta_L\Delta)], \ast)$ is a group follows immediately from our choice of $\Delta$, as 
		$\Delta_{BCH}$ is the common denominator of the coefficients of the Baker-Campbell-Hausdorff formula. This ends the proof.
	\end{proof}
	
	\section{Geometry of $\mathcal{M}$-Groups} \label{sec_geom}
	Let $G$ be an $\mathcal{M}$-group with the fixed generating set of Notation \ref{nota_basis}. By Lemma \ref{lem_g_as_product}, we know that if $g\in B_G(r)\cap K$, we can write $g$ as a formal product of the form $k_1^{z_1}\dots k_m^{z_m}$ with $z_i \in \Z[1/\Delta]$. Yet, we do not have a relation between the size of $z_i$ and $r\in \R_{\geq 1}$. The goal of this section is to prove the following statements that we will use in section \ref{sec_upper}:
	\begin{itemize}
		\item If $g\in K\cap B_G(r)$, then we can write $z_i$ as $\mu_i/\Delta^{j_i}$ with $\mu_i\in\Z$ and $j_i \in \N\cup\{0\}$ such that $|\mu_i| \leq C^r$ for some fixed $C>1$.
		\item If $|H|< \infty$ or equivalently if $G$ is virtually nilpotent, then $|\mu_i|$ is bounded by a polynomial in $r$.
	\end{itemize}
	We have split the proof of this fact in two parts. In subsection \ref{sec_bounding_coeff}, we will use the group setting (Notations \ref{nota_group}-\ref{nota_poly}) to bound the coefficients $z_i$. In subsection \ref{sec_bounding_denom}, we will use the Lie setting (Notations \ref{nota_Lie}-\ref{nota_LR_groups}) to bound the denominators $\Delta^{i_j}$. This yields a bound for $\mu_i$, since $|\mu_i| = |z_i|\cdot |\Delta^{i_j}|$.
	\subsection{Bounding coefficients} \label{sec_bounding_coeff}
	In this subsection, we will focus on the proof of Theorem \ref{thm_geometry} below, in which we will bound $|z_i|$. Note that such a bound has already been established in some special cases, for example when $G = K$, in \cite{MR379672,MR623534,Wolf}. Our proofs give a generalization of these techniques to the case where $K$ is not necessarily finitely generated.
	\begin{thm} \label{thm_geometry}
	Take notations as in Notation \ref{nota_group} and  \ref{nota_basis}.
		\begin{enumerate}[(i)]
			\item There exists a constant $C>0$ such that $g \in B_G(r)\cap K$ implies that $g = k_1^{z_1}\dots k_m^{z_m}$ with $|z_i| \leq C^r$.
			\item If $|H|< \infty$, then there exists $C' > 0$ such that $g \in B_G(r)\cap K$ implies that $g = k_1^{z_1}\dots k_m^{z_m}$ with $|z_i| \leq C'r^{m^m}$.
		\end{enumerate} 
	\end{thm}

	\noindent	The statements will be proven in Proposition \ref{prop_metric} and Corollary \ref{cor_metric_minimax} respectively.

	\begin{lemma} \label{lem_action_geom}
		Take notations as in Notation \ref{nota_group} and  \ref{nota_basis}. There exists a constant $C_1>0$ such that every element $g\in B_{\bar{G}}(r)\cap K$ can be written as a product
		$$g = \prod_{s=1}^d g_s $$
		with $g_s \in \{k_i^{z}\mid z\in [-1,1]\cap\Z[1/\Delta], 1\leq i \leq m\}$ and $d \leq C_1^r$.
	\end{lemma}
	\begin{proof}
		We start by defining two integers $A_1, A_2>0$. Given a generator $k_i$ and an automorphism $\xi_j^\epsilon$ with $\epsilon \in \{1,-1\}$, it holds that $\xi_j^\epsilon(k_i^x) = k_1^{z_1}\dots k_m^{z_m}$ where every $z_i$ is a polynomial in $x$. Hence, the function $|z_i|$ is bounded on the interval $[-1,1]$, and thus there exists a universal upper bound $\tilde{A_1}$ for all these $z_i$:
		$$\tilde{A_1} = \sup\{|z_1|, \ldots, |z_m| \mid \xi_j^\epsilon(k_i^x) = k_1^{z_1}\dots k_m^{z_m}\text{ with } x\in [-1,1], \epsilon\in\{1,-1\}, 1\leq i \leq m, 1\leq j \leq n\}.$$ A constant $\tilde{A_2}$ is similarly defined as follows:
		$$\tilde{A_2} = \max\{|z_1|, \ldots, |z_m| \mid h_i^{\epsilon_i}h_j^{\epsilon_j} = k_1^{z_1}\dots k_m^{z_m}h_j^{\epsilon_j}h_i^{\epsilon_i}\text{ with }1\leq i,j \leq n, \, \epsilon_i, \epsilon_j \in \{1,-1\}\}.$$
	
		Set $A_1 = \lceil \tilde{A_1} \rceil$ and $A_2 = \lceil \tilde{A_2} \rceil$. By the definition of $A_1$, we know that if $k_i^x$ with $x\in [-1,1]$ is given then
		 \begin{equation} \label{eq_rotat_gen_1}
		 		h_j^{\pm 1}\cdot k_i^{x} = k_1^{z_1}\dots k_m^{z_m} \cdot h_j^{\pm 1},
		 \end{equation}
	 	where $k_l^{z_l}$ can be written as a product of at most $A_1$ elements of the form $k_l^y$ with $y\in [-1,1]$, since $|z_l| \leq A_1$. In particular, going from the left to the right hand side introduces at most $mA_1$ elements of such a form. Analogously, the definition of $A_2$ yields a similar effect for switching elements $h_i^{\epsilon_i}$ and $h_j^{\epsilon_j}$.
	 	
	 	Now, set $A = \max\{mA_1,A_2\}$, and take $g\in B_{\bar{G}}(r)\cap K$. We know that this element can be rewritten to $g = kh_1^{l_1}\dots h_n^{l_n}$ with $k\in K$. Since $g\in K$, we know that $l_1 = \ldots = l_n = 0$. We claim that $k$ can be constructed as a product of at most $(A+1)^rr$ elements of the form $k_i^{z_i}$ with $|z_i| \leq 1$, which implies the statement of the lemma.
	 	
	 	Since $g\in B_{\bar{G}}(r)\cap K$, it is given by a product of at most $r$ generators of $\bar{G}$. In particular, at most $r$ factors are of the form $h_j^{\pm 1}$. Take a factor of the form $h_1^{\pm 1}$ (if there are any), then our observation shows that moving this element one position to the right introduces at most $A$ elements of the form $k_i^{z_i}$ with $|z_i| \leq 1$. Moving this element to the right-most position hence introduces at most $A r$ elements of the given form, leaving a product with in total at most $(A+1)r$ elements of this form. 
	 	
	 	Proceeding this way with a second element of the form $h_1^{\pm 1}$ introduces at most $A(A+1)r$ new elements, leaving $A(A+1)r + (A+1)r = (A+1)^2r$ elements in total. We can continue this process, first for $h_1^{\pm 1}$, then for $h_2^{\pm 1}$, etc. As there are at most $r$ generators of $\{h_j^{\pm 1}\mid 1\leq j \leq n\}$, this gives at most $(A+1)^rr$ elements of the form $k_i^{z_i}$ with $|z_i| \leq 1$.
	\end{proof}
	The previous lemma gives a formal product of the form
	$$g = \prod_{s=1}^d g_s $$
	with $g_s \in \{k_i^{z}\mid z\in [-1,1]\cap\Z[1/\Delta], 1\leq i \leq m\}$ and $d \in \N$, which can be rewritten to the form $k_1^{z_1}\dots k_m^{z_m}$ with $z_i \in \Z[1/\Delta]$. The next result gives an estimate for $|z_i|$ in terms of $d\in \N$, using the following notation:
	\begin{df}\label{df_degree}
		Given a formal product $g$ of the form
		$$g = \prod_{s=1}^d g_s $$
		with $g_s \in \{k_i^{z}\mid z\in [-1,1], 1\leq i \leq m\}$ and $d \in \N$. We define the degree $\degr(g)$ as the vector $(x_1, \ldots, x_m)$, where $x_i$ is the number of times a factor of the form $k_i^{z}$ ($-1\leq z\leq 1$) appears.
	\end{df}
	\begin{lemma} \label{lem_to_basis}
		There exists a constant $C_2>0$ such that if $g$ is a formal product as in Definition \ref{df_degree} with
		$$\degr(g) \leq  (0, \ldots,0,r_i, \ldots, r_m) \quad \text{where } r_i = \ldots = r_m,$$
		then $g$ can be rewritten to $k_i^{z_i}g'$ where $|z_i| \leq r_i$ and $g'$ is a formal product with
		$$\degr(g') \leq  (0, \ldots,0,r'_{i+1}, \ldots, r'_m) \quad \text{where } r'_{i+1}= \ldots = r'_m \leq C_2r_i^m.$$
	\end{lemma}
	\begin{proof}
		We start by introducing a constant $A>0$. For this, recall that by Proposition \ref{prop_malcev_poly}, for all $x,y\in \Q$ and $i<j$, we have
		$$k_j^yk_i^x = k_i^xk_j^y\cdot k_{j+1}^{z_{j+1}}\ldots k_m^{z_m},$$
		where all $z_l$ are polynomials in $x$ and $y$. Since polynomials are bounded on compact subsets, we can define
		$$\tilde{A} = \sup\{|z_{j+1}|, \ldots , |z_m| \mid k_j^yk_i^x = k_i^xk_j^y\cdot k_{j+1}^{z_{j+1}}\ldots k_m^{z_m}\text{ with }x,y\in [-1,1], \, 1\leq i < j \leq m \}.$$
		Set $A = \lceil \tilde{A} \rceil$.
		
		Now, consider the formal product $g$. Take, if possible, a factor of the form $k_i^x$ with $|x| \leq 1$. If we want to move this factor one position to the left, it needs to switch places with an element of the form $k_j^y$ with $i\leq j$ and $y\in [-1,1]$. By the definition of $A$, this switch introduces at most $A$ new factors of the form $k_l^z$ with $|z| \leq 1$ for every $l>i$. In particular, by shifting the first occurrence of an element of the form $k_i^x$ with $|x| \leq 1$ to the front of the product, we obtain a new word $\tilde{g}$ of at most degree
		\begin{eqnarray*} \label{eq_matrix_switch}
			\degr(\tilde{g}) \leq \begin{pmatrix}0\\ \vdots \\ 0 \\ r_i \\ r_{i+1} \\ r_{i+2} \\ \vdots \\ r_m \end{pmatrix} + \begin{pmatrix}0\\ \vdots \\ 0 \\ 0 \\ 0 \\ Ar_{i+1} \\ \vdots \\ A(r_{i+1} + \ldots + r_{m-1}) \end{pmatrix} 
			&\leq& \begin{pmatrix}0\\ \vdots \\ 0 \\ r_i \\ r_{i+1} \\r_{i+2} \\ \vdots \\ r_m \end{pmatrix} + \begin{pmatrix}0\\ \vdots \\ 0 \\ 0 \\ Ar_{i} \\ A(r_i + r_{i+1}) \\ \vdots \\ A(r_i + \ldots + r_{m-1}) \end{pmatrix} \\ &=& \begin{pmatrix}0&\hdots & 0 && & & \\ \vdots & \ddots & \vdots &&&& \\ 0& \hdots & 0 &&&& \\ &&& 1& 0 &\hdots & 0 \\ &&&A & \ddots & \ddots & \vdots \\ &&&\vdots&\ddots &\ddots&0\\ &&&A&\hdots &A&1 \end{pmatrix}\begin{pmatrix}0\\ \vdots \\ 0 \\ r_i \\ r_{i+1} \\ \vdots \\ r_m \end{pmatrix}.
		\end{eqnarray*}
		
		Now, repeat this process until all factors $k_i^x$ are in front. Since there are no more than $r_i$ such factors, we conclude that $g$ can be rewritten to $k_i^{z_i}g'$ with $|z_i| \leq r_i$ and $g'$ a formal product with degree:
		\begin{equation}\label{eq_estimate_kprime}
			\degr(g') \leq \begin{pmatrix}0&\hdots & 0 && & & \\ \vdots & \ddots & \vdots &&&& \\ 0& \hdots & 0 &&&& \\ &&& 1& 0 &\hdots & 0 \\ &&&A & \ddots & \ddots & \vdots \\ &&&\vdots&\ddots &\ddots&0\\ &&&A&\hdots &A&1 \end{pmatrix}^{r_i}\begin{pmatrix}0\\ \vdots \\ 0 \\ r_i \\ r_{i+1} \\ \vdots \\ r_m \end{pmatrix}.
		\end{equation}
		Observe that the matrix has a block structure of the form $\begin{psmallmatrix} 0&0\\0& \mathbb{1} + N\end{psmallmatrix}^{r_i}$, where $N$ is nilpotent and hence
		$$(\mathbb{1} + N)^{r_i} = \mathbb{1} + \binom{r_i}{1}N + \ldots \binom{r_i}{m-1}N^{m-1} + 0.$$
		Therefore, the entries of this matrix can be estimated by a polynomial of the form $Br_i^{m-1}$ for some constant $B>0$ depending on $A$ and $m$. Using this estimate in Equation \eqref{eq_estimate_kprime} shows that the claimed constant $C_2>0$ from the lemma's statement surely exists.   
	\end{proof}
	\begin{prop} \label{prop_metric}
		Take notations as in Notation \ref{nota_group} and  \ref{nota_basis}, then there exists a constant $C>0$ such that $g \in B_{\bar{G}}(r)\cap K$ implies that $g = k_1^{z_1}\dots k_m^{z_m}$ with $|z_i| \leq C^r$. Furthermore, if $K = \bar{K}$, then there exists a constant $C'>0$ such that $g\in B_K(r)$ implies that $g = k_1^{z_1}\dots k_m^{z_m}$ with $|z_i| \leq C'r^{m^m}$.
	\end{prop}
	\begin{proof}
		For convenience in this proof, we will write $|z_i| = O(r^k)$ to indicate that $|z_i|$ can be estimated by some polynomial in $r$ of degree $k$, where the coefficients depend only on $K^\Q$.
		
		Suppose a formal product $g$ as in Definition \ref{df_degree} is given with $\deg(k) \leq (r, r , \ldots)$. Applying Lemma \ref{lem_to_basis}, we can rewrite $k$ to $k_1^{z_1}g'$, where 
		$|z_1| = O(r)$ and $\degr(g') = O(r^m)$. Here, $g'$ has no factor of the form $k_1^z$ with $|z|\leq 1$. Now, applying this lemma again and again shows that $g = k_1^{z_1}\dots k_m^{z_m}$ with $|z_2| = O((r^m)^m) = O(r^{m^2})$ and more generally $|z_m| = O(r^{m^m})$.
		
		Suppose now that we have $g\in B_{\bar{G}}(r)\cap K$. By Lemma \ref{lem_action_geom}, we know that 
		$$\degr(g) \leq (C_1^r, C_1^r, \ldots),$$ and thus if we rewrite it to the form $k_1^{z_1}\dots k_m^{z_m}$, then $|z_i| = O((C_1^r)^{m^m}) = O(C_1^{m^mr})$ for all $1\leq i \leq m$. Hence, some exponential upper bound of the form $C^r$ must exist.
		
	For the final part, assume that $K = \bar{K}$. In particular, $K$ is finitely generated with Mal'cev basis $\{k_1, \ldots , k_m\}$. If $g\in B_K(r)$, with respect to the generators $\{k_1 , \ldots , k_m\}$, then clearly $\degr(g) \leq (r,r, \ldots)$. We conclude via the argumentation above.
	\end{proof}
	We end this section by extending the results from the previous proposition about $\bar{G}$ to $G$ itself.
	\begin{cor} \label{cor_metric_minimax}
		Take notations as in Notation \ref{nota_group} and  \ref{nota_basis}, then there exists a constant $C>0$ such that $g \in B_{G}(r)\cap K$ implies that $g = k_1^{z_1}\dots k_m^{z_m}$ with $|z_i| \leq C^r$. Furthermore, if $|H|< \infty$, then there exists $C' > 0$ such that $g \in B_G(r)\cap K$ implies that $g = k_1^{z_1}\dots k_m^{z_m}$ with $|z_i| \leq C'r^{m^m}$.
	\end{cor}
	\begin{proof}
		Since $\bar{G} \lhd_f G$, the inclusion map $\bar{G} \to G$ is bi-Lipschitz, hence there exists a constant $A>0$ such that $B_G(r)\cap \bar{G} \subset B_{\bar{G}}(Ar)$. From this, the result is immediate as in the final part, $|H|<\infty$ implies that $\bar{G} = K = \bar{K}$.
	\end{proof}

	\subsection{Bounding denominators} \label{sec_bounding_denom}
	In this subsection, we will complete the claim made at the beginning of the section, by showing that if $g\in B_G(r)\cap K$ and $g = k_1^{z_1}\cdots k_m^{z_m}$ with $z_i = \mu_i/\Delta^{j_i}$, then $|\mu_i|$ can be exponentially bounded. Since we can already bound $|z_i|$, this subsection will focus on bounding the denominator $|\Delta^{j_i}|$. Note that the denominator is not uniquely determined, so the claim is that a small enough denominator can be chosen.
	
	In light of section \ref{sec_upper}, we will only prove the result indirectly: we will prove the statement in the additive notation of the Lie ring $L$. The claim in $K$ itself then follows by Proposition \ref{prop_coor}. We start by making a relevant observation concerning the Baker-Campbell-Hausdorff formula.
	
	\begin{lemma}
		\label{lem:longprod}
		There exists a power of $\Delta$ that is a common denominator of the rational coefficients in the formal expressions
		$\{w_1 \ast w_2 \ast \ldots \ast w_l \mid l\in\N\}$,
		where $\ast$ is the Baker-Campbell-Hausdorff formula for a nilpotent group of nilpotency class $c$.
	\end{lemma}
	\begin{proof}
		We first show that a common denominator of all formal expressions $\{w_1 \ast w_2 \ast \ldots \ast w_l \mid l\in\N\}$ exists. Consider
		$$w_1\ast w_2 =  w_1+w_2 + \frac{1}{2}[w_1,w_2]_L + \sum_{e=3}^{c} q_e(w_1,w_2). $$
		Let $d_1 = 2$ and $d_{e-1}$ be a common denominator of the rational coefficients in $q_e(w_1,w_2)$. Write $S_2 = \{d_i\mid 1\leq i \leq c-1\}$ for the set of all denominators different from $1$. 
		
		Consider the expression $w_1\ast w_2 \ast w_3$. We have
		\begin{equation*}
			\begin{split}
				(w_1\ast w_2)\ast w_3 & =w_1\ast w_2+w_3 + \frac{1}{2}[w_1\ast w_2,w_3]_L + \sum_{e=3}^{c} q_e(w_1\ast w_2,w_3) \\
				& = w_1\ast w_2+w_3 + \frac{1}{2}[w_1,w_3]_L + \frac{1}{2}[w_2,w_3]_L + \frac{1}{2}\frac{1}{2}[ [w_1,w_2]_L,w_3]_L + \ldots
			\end{split}
		\end{equation*}
	One sees that the set of occurring denominators different from $1$ in this expression is contained in the finite set $S_3 = S_2 \cup \{d_id_j^k \mid 1\leq i,j \leq c-1, k \leq c-1\}$, as the length of non-zero brackets is at most $c$.

	Repeating this argument for $w_1\ast w_2\ast w_3 \ast w_4$ and longer products shows that in general for $l\in \N$ the set of occurring denominators not equal to 1 of $w_1\ast w_2\ast \ldots \ast w_l$ is contained in
	$$S_{c} :=  \{d_{i_1}^{k_{i_1}}d_{i_2}^{k_{i_2}}\cdots d_{i_l}^{k_{i_l}} \mid k_{i_j} \geq 0, 1\leq \sum_{j=1}^l k_{i_j} \leq c \}.$$
		Since $\Delta$ was a multiple of all numbers in $S_2$, some power of $\Delta$ must be a multiple of all numbers in $S_c$. This ends the proof.
	\end{proof}
	\begin{prop} \label{prop_Br_to_basis}
	Take notations as in Notation \ref{nota_LR_groups}, with $\{v_1, \ldots , v_m\}$ a $\Z$-basis of $L$. There exists a constant $C>0$ such that if $g\in B_G(r)\cap K$, then 
		$$g = \sum_{i=1}^m \lambda_iv_i$$
		with $\lambda_i = \mu_i/\Delta^{j_i}$ for some $j_i\in\N$ and some $\mu_i\in \Z$ satisfying $|\mu_i| \leq C^r$. If $H$ is finite, then $|\mu_i| \leq Cr^C$.
	\end{prop}
	\begin{proof}
		Recall that $K \subset L^\Delta$, so surely every element $g\in B_G(r)\cap K$ can be written in the form $g = \sum_{i=1}^m \lambda_iv_i$ with $\lambda_i = \mu_i/\Delta^{j_i}$ for some $j_i\in\N$ and some $\mu_i\in \Z$. Hence, it suffices to argue that $|\mu_i| \leq C^r$ for some fixed constant $C>0$. Note that $|\mu_i| = |\lambda_i|\cdot |\Delta^{j_i}|$, so it suffices to argue that we can assume both factors are exponentially bounded.
		
		By Theorem \ref{thm_geometry}, we know that $g = k_1^{\tilde{z}_1}\cdots k_m^{\tilde{z}_m}$
		with $\tilde{z}_i \in \Z[1/\Delta]$ and $|\tilde{z}_i| \leq \tilde{C}^r$ for some fixed constant $\tilde{C} >0$. Now, Proposition \ref{prop_coor} allows us to rewrite $g$ in the Lie algebra $\mathfrak{k} \cong L^\Q$ to the form
		$$g = z_1k_1+ \ldots + z_mk_m .$$
		Since the coordinates $z_i$ are fixed polynomials in $\{\tilde{z}_i\mid 1\leq i \leq m\}$, we observe that $|z_i|$ is also exponentially bounded. Using a linear transformation to the fixed basis $\{v_1, \ldots , v_m\}$ shows that $g = \sum_{i=1}^m \lambda_iv_i$, where $|\lambda_i| \leq C_1^r$ for some fixed $C_1>0$.	In order to show that $\Delta^{j_i}$ can be chosen such that $|\Delta^{j_i}| \leq C_2^r$ for some fixed $C_2>0$, we recall that $B_G(r) \cap K \subset B_{\bar{G}}(C_3r)\cap K$ for some constant $C_3>0$ (since $\bar{G} \leq_f G$). Hence, it suffices to show the claim for elements $g$ in $B_{\bar{G}}(r)\cap K$.
		
		Consider the finite set of elements $$S = \{k_i, k_i^{-1}\mid 1\leq i \leq m\}\cup\{k \in K\mid 1\leq i,j\leq n: \epsilon_i, \epsilon_j \in \{1,-1\}: h_i^{\epsilon_i}h_j^{\epsilon_j} = kh_j^{\epsilon_j}h_i^{\epsilon_i}\}.$$
		There surely exists some $n_1\in \N$ such that every element in this set can be written as $\sum_{i=1}^m \lambda_iv_i$ with $\lambda_i = \mu_i/\Delta^{n_1}$ with $\mu_i\in \Z$. Also, recall that $h_jk = \xi_j(k)h_j$. Take $\Delta^{n_2}$ for $n_2\in \N$ to be a common denominator of the entries of all matrices $\{\xi_j , \xi_j^{-1}\mid 1\leq j \leq n\}$ with respect to the chosen basis of $L$. (Note that $\xi_k^{\pm 1}: L^\Delta \to L^\Delta$ is well-defined, and thus such a common denominator exists.)
		
		Take $g\in B_{\bar{G}}(r)\cap K$. This element is a product of the form 
		$$g = g_1 g_2 \cdots g_r $$
		with $g_l \in \{k_i, k_i^{-1}, h_j, h_j^{-1}\mid 1\leq i \leq m, 1 \leq j \leq n\}$. Move all elements of type $h$ (starting with elements of the form $h_n^{\pm 1}$) to the right. We obtain an expression of the form
		$$g = \tilde{g}_1 \cdots \tilde{g}_s h_1^{l_1} \cdots h_n^{l_n},$$
		where $s\in \N$, $l_j \in \N$ and $\tilde{g}_i$ is of the form $\xi(k)$ with $\xi$ a composition of at most $r$ homomorphisms in $\{\xi_j , \xi_j^{-1}\mid 1\leq j \leq n\}$ and $k\in S$. Since $g\in K$, we know in fact that $l_1 = \ldots = l_n = 0$.
		
		By the choice of $n_1,n_2\in \N$, we see that (for every $1\leq i \leq s$)
		$$\tilde{g}_i = \sum_{i=1}^m \lambda_iv_i$$
		with $\lambda_i \in (1/\Delta^{n_1 + rn_2})\Z$. Now, we wish to apply Lemma \ref{lem:longprod} to the product $\tilde{g}_1 \cdots \tilde{g}_s$, where the size of $s$ does not matter, leading to a denominator of $\Delta^{n_3}$ for some $n_3 \in \N$.
		The Lie bracket $[\cdot , \cdot]_L$ has integral structure constants on $L$, so if vectors $w_i$ have coordinates over $\Z$, then every Lie bracket still has coordinates in $\Z$. However, in our case, the coordinates lie in $(1/\Delta^{n_1 + rn_2})\Z$. Using linearity, this implies that the repeated Lie bracket has coordinates in $(1/\Delta^{c(n_1 + rn_2)})\Z$, as the length of a repeated Lie bracket is bounded by $c$. We thus conclude that
		$$g = \sum_{i=1}^m \lambda_iv_i$$
		with $\lambda_i \in (1/\Delta^{n_3+ c(n_1 + rn_2)})\Z$. This ends the first part.
		
		Now, suppose $|H|< \infty$, then we know by construction that $\bar{G} = K$ and $K = \bar{K}$. In particular, $g = k_1^{\tilde{z}_1}\cdots k_m^{\tilde{z}_m}$ with $\tilde{z}_i \in \Z$. By Theorem \ref{thm_geometry}, $|\tilde{z}_i|$ is polynomially bounded in $r$. Rewriting this to $g = z_1k_1+ \ldots + z_mk_m$ using Proposition \ref{prop_coor} shows that $z_i \in (1/N_1)\Z$, where $N_1\in\N$ is the common denominator of the polynomials governing this rewriting process. The coordinate $z_i$ is still polynomially bounded. Now, using a linear transformation to the fixed basis $\{v_1, \ldots , v_m\}$ shows that $g = \sum_{i=1}^m \lambda_iv_i$, where $|\lambda_i|$ is polynomially bounded in $r$. The denominator of $\lambda_i$ is bounded by $N_1N_2$, where $N_2\in\N$ is the common denominator of the matrix entries representing the linear transformation. From this, the statement follows.
	\end{proof}
	\section{Upper Bound} \label{sec_upper}
	In this section, we will first focus on constructing normal subgroups in $G$, by relating ideals in $L^\Delta$ to normal subgroups in $K$ itself. Afterwards we apply this to prove the upper bound in Theorem \ref{thm_upper}. The notations were introduced in Section \ref{sec_nota}. 
	
	\subsection{Normal subgroups and ideals}
	In the next proofs, we will show that under suitable circumstances ideals in $L$ and $L^\Delta$ are also normal subgroups of $L$ and $L^\Delta$. In essence, this will be a generalization of the following result in \cite[Lemmata 4.6-4.8]{grunewald1988subgroups} to the case of non-finitely generated groups.
	\begin{lemma} \label{lem_sega_normal_ideal}
		Let $L$ be a finitely generated nilpotent Lie ring, such that $(L, \ast)$ is a group. There exists a constant $M>0$ such that for all prime power $p^k$ with $p>M$ the ideals $I$  of index $p^k$ are exactly the normal subgroups of index $p^k$.
	\end{lemma}
\noindent	Recall that the Baker-Campbell-Hausdorff formula dictates that $\lambda v$ in the Lie algebra equals $v^\lambda$ in the group (for all $\lambda \in \Q$).
	\begin{lemma} \label{lem_iso_K_L}
		Take notations as in Notation \ref{nota_LR_groups} and the bound $M$ of Lemma \ref{lem_sega_normal_ideal}. For any prime $p>\max\{\Delta, M\}$, the inclusion map $L\hookrightarrow L^\Delta$ induces for every $k\in \N$ a Lie ring isomorphism 
		$$\dfrac{L}{p^kL} \cong \dfrac{L^\Delta}{p^kL^\Delta}.$$
		Furthermore, both $p^kL$ and $p^kL^\Delta$ are normal subgroups, and the inclusion maps $L\hookrightarrow K \hookrightarrow L^\Delta$ induce group isomorphisms
		$$\dfrac{L}{p^kL} \cong \dfrac{K}{K\cap p^kL^\Delta}\cong \dfrac{L^\Delta}{p^kL^\Delta}.$$
	\end{lemma}
	\begin{proof}
		It is clear that $p^kL$ and $p^k L^\Delta$ are ideals, by the linearity of the Lie bracket. The set $p^kL$ is a normal subgroup by Lemma \ref{lem_sega_normal_ideal}. The set $p^kL^\Delta$ is a subgroup, as the Baker-Campbell-Hausdorff formula (for $v_1, v_2 \in L^\Delta$) implies that
		$$p^kv_1 \ast (p^kv_2)^{-1} = p^kv_1 \ast (-p^kv_2) = p^kv_1-p^kv_2 - p^{2k}\left(\frac{1}{2}[v_1,v_2]_L\right) + \sum_{e=3}^\infty p^{ek}q_e(v_1,-v_2),$$
		where we have coefficients over $\Z[1/\Delta]$ as $ p > \Delta$, hence every term of this expression lies in $p^kL^\Delta$. It is also normal as the commutator $[v_1,v_2] = v_1^{-1}v_2^{-1}v_1v_2$ is equal to 
		$$[v_1,v_2] = [v_1,v_2]_L + \sum r_ic_i, $$
		where $r_i \in \Z[1/\Delta]$ and $c_i$ are repeated Lie brackets containing both $v_1$ and $v_2$, see \cite[Chap. 6, Cor. 2-3]{sega83-1}.
	
		Now consider the inclusion map $i: L \hookrightarrow L^\Delta$, which induces both a group morphism $L \to L^\Delta/p^kL^\Delta$ as an algebra morphism $L \to L^\Delta/p^k L^\Delta$ . Surjectivity of these morphisms is clear as $\Delta$ is invertible over $p^k$ by our assumption. 
		Since the kernels are $p^kL$, this shows the claim about the isomorphism for the Lie algebras and the group isomorphism
		$$\dfrac{L}{p^kL} \cong \dfrac{L^\Delta}{p^kL^\Delta}.$$
		However, from the inclusions $L\hookrightarrow K \hookrightarrow L^\Delta$, it is then immediate that this extends to isomorphisms
		$$\dfrac{L}{p^kL} \cong \dfrac{K}{K\cap p^kL^\Delta}\cong \dfrac{L^\Delta}{p^kL^\Delta}.$$
	\end{proof}

	\begin{lemma}\label{lem_normal_ideal}
		Use Notation \ref{nota_LR_groups} and the bound $M$ of Lemma \ref{lem_sega_normal_ideal}. If $p^k$ is a prime power with $p>\max\{\Delta, M\}$, then $I^\Delta$ is an ideal of $L^\Delta$ of index $p^k$ if and only if $I^\Delta$ is a normal subgroup of $L^\Delta$ of index $p^k$. Furthermore, if $I^\Delta$ is an $\mathcal{H}$-invariant ideal of $L^\Delta$ of index $p^k$, then $I^\Delta\cap K$ is an $\mathcal{H}$-invariant normal subgroup of $K$ of index $p^k$.
	\end{lemma}
	\begin{proof}
		By Lemma \ref{lem_sega_normal_ideal}, ideals and subgroups of index $p^k$ are the same subsets of $L$. Let $I$ denote such an ideal and consider the surjective Lie ring morphism $\pi_L: L \to L^\Delta/p^kL^\Delta$ and the surjective group morphism $\pi_G: L \to L^\Delta/p^kL^\Delta$. 		Since both are surjective, they map ideals to ideals and normal subgroups to normal subgroups respectively. Write $I_L = \pi_L^{-1}(\pi_L(I)) = I+p^kL^\Delta\subset L^\Delta$ for the ideal and $I_G = \pi_G^{-1}(\pi_G(I)) = I\ast p^kL^\Delta \subset L^\Delta$ for the normal subgroup, then we will show that $I_L = I_G = I\otimes_\Z \Z[1/\Delta]$.
		
		Firstly, take an arbitrary element in $I\otimes_\Z \Z[1/\Delta]$. This element is of the form $(1/\Delta^l)v$ with $l\in\N$ and $v\in I$. Take $e\in \N$ such that $(e\Delta)^l = 1 + zp^k$ for some $z\in \Z$. Now,
		$$\dfrac{1}{\Delta^l}v = e^lv - p^k \left(\dfrac{z}{\Delta^l}v\right) = e^lv \ast p^k \left(\dfrac{-z}{\Delta^l}v\right).$$
		We conclude that $I\otimes_\Z \Z[1/\Delta] \subset I_G$, and a similar arguments holds for $I_L$ as well. Secondly, note that $I\otimes_\Z \Z[1/\Delta]$ is additively and multiplicatively closed. Indeed, for multiplicativity, since the Baker-Campbell-Hausdorff formula has coefficients in $\Z[1/\Delta]$, we can rewrite the product of two arbitrary elements $(1/\Delta^{l})v_1$ and $(1/\Delta^{l})v_2$ in $I\otimes_\Z \Z[1/\Delta]$ with $v_1,v_2\in I$ to
		$$\dfrac{1}{\Delta^l}v_1\ast \dfrac{1}{\Delta^l}v_2 = \dfrac{1}{\Delta^l}v_1+ \dfrac{1}{\Delta^l}v_2+ \frac{1}{2}\dfrac{1}{\Delta^{2l}}[v_1,v_2]_L + \sum_{e=3}^\infty \dfrac{1}{\Delta^{el}}q_e(v_1,v_2),$$
		i.e.~a $\Z[1/\Delta]$-linear combination of elements in $I$. Moreover, an arbitrary element $p^k(1/\Delta^l)v$ of $p^kL^\Delta$ with $v\in L$ is equal to $(1/\Delta^l)(p^kv)$, and $p^kv \in I$ since $[L:I] = p^k$. Hence, $p^k L^\Delta \subset I\otimes_\Z \Z[1/\Delta]$ and as also $I \subset I\otimes_\Z \Z[1/\Delta]$, we conclude that $I_L, I_G \subset I\otimes_\Z \Z[1/\Delta]$. This shows the claim that $I_G = I_L = I\otimes_\Z \Z[1/\Delta]$.
	
Now to prove the lemma, let $I^\Delta$ be an ideal of index $p^k$ in $L^\Delta$ and take $I = I^\Delta \cap L$, so $I^\Delta = I_L$. Now, $I$ is an ideal of $L$ of index $p^k$ by the isomorphisms
		$$\dfrac{L}{I} \cong \dfrac{L/p^kL}{I/p^kL} \cong \dfrac{L^\Delta/p^kL^\Delta}{I^\Delta/p^kL^\Delta} \cong \dfrac{L^\Delta}{I^\Delta}.$$
		Hence, $I$ is also a normal subgroup of index $p^k$. By the same isomorphism interpreted over groups, $I_G$ is a normal subgroup of $L^\Delta$ with the same index. Note that $I^\Delta = I_L = I_G$. This ends one direction of the statement. The other direction is completely analogous.
		
		For the `furthermore' part, observe that the intersection of invariant subspaces is forcefully invariant itself. The fact that $I^\Delta\cap K$ is normal in $K$ with index $p^k$ follows from the isomorphism $K/(K\cap p^kL^\Delta)\cong L^\Delta/p^kL^\Delta$.
	\end{proof}
In the previous result, we have seen that normal subgroups in $K$ can be constructed from ideals in $L^\Delta = L \otimes_\Z \Z[1/ \Delta]$. In the remainder of this section, we will focus on a particular subclass of ideals in $L^\Delta$, namely those that correspond to ideals in $L^{\Z_p} = L/pL$ for primes $p$.
	\begin{lemma} \label{lem_LDelta_LZp}
	Take notations as in Notation \ref{nota_Lie} and let $p$ be a prime larger than $\max\{M, \Delta\}$ as in Lemma \ref{lem_sega_normal_ideal}. Consider the morphism of Lie rings
		$$\psi : L^\Delta \to \dfrac{L^\Delta}{pL^\Delta} \cong L^{\Z_p}.$$
		If $pL^\Delta \leq I^\Delta$ is an $\mathcal{H}$-invariant ideal of index $p^k$, then $\psi(I^\Delta)$ is an $\mathcal{H}$-invariant ideal in $L^{\Z_p}$ of index $p^k$. Vice versa, if $J$ is an $\mathcal{H}$-invariant ideal in $L^{\Z_p}$ of index $p^k$, then $\psi^{-1}(J)$ is an $\mathcal{H}$-invariant ideal in $L^\Delta$ of index $p^k$. 
	\end{lemma}
	\begin{rem}
		Here, invariance under $\xi \in \mathcal{H}$ in $L^{\Z_p}$ is understood as invariance under the induced action of $\xi$ on $L^{\Z_p}$, i.e. under the homomorphisms $\bar{\xi}$ such that the following diagram commutes
		\[\begin{tikzcd}
			L^\Delta & L^{\Z_p} \\
			L^\Delta & L^{\Z_p}.
			\arrow["\psi", from=1-1, to=1-2]
			\arrow["\xi"', from=1-1, to=2-1]
			\arrow["\bar{\xi}", from=1-2, to=2-2]
			\arrow["\psi"', from=2-1, to=2-2]
		\end{tikzcd}\]
		If one were to take a basis of $L$, then we know that the matrix representing $\xi$ has entries over $\Z[1/\Delta]$. Now, $\bar{\xi}$ corresponds to the matrix of $\xi$ where the projection $\Z[1/\Delta] \to \Z_p$ is applied to its entries.
	\end{rem}
	\begin{proof}
	Since $\psi$ is surjective, ideals are preserved under taking their image or their preimage. Since $pL^\Delta \leq I^\Delta$, we have
		$$[L^\Delta:I^\Delta] = [pL^\Delta: pL^\Delta\cap I^\Delta]\cdot [L^{\Z_p}: \psi(I^\Delta)] = [L^{\Z_p}: \psi(I^\Delta)],$$
		which shows the claim about the indices. The claim about the $\mathcal{H}$-invariance is immediate.
	\end{proof}

	\subsection{Proof of the upper bound}
	
	Let us first introduce the value $\delta(\mathfrak{k}^{\bar{\Q}}, \mathcal{H})$ that will appear in the upper bound of $\RF_G$. Write the algebraic closure of $\Q$ by $\bar{\Q}$.

	\begin{df}
	Let $\mathfrak{k}^{\bar{\Q}}$ be a Lie algebra over $\bar{\Q}$. Suppose $\mathcal{H}$ is a finite set of automorphisms of $\mathfrak{k}^{\bar{\Q}}$. Define
	\begin{equation*}
		\delta(\mathfrak{k}^{\bar{\Q}}, \mathcal{H}) = \min\{\max_{i=1}^k\{\dim_{\bar{\Q}}(\mathfrak{k}/I_i^{\bar{\Q}})\}\mid I_1^{\bar{\Q}} \text{ to } I_k^{\bar{\Q}} \text{ are }\mathcal{H}\text{-invariant ideals of }\mathfrak{k}^{\bar{\Q}}, \;\cap_{i=1}^k I_i^{\bar{\Q}} = \{0\}\}.
	\end{equation*}
	\end{df}
	\begin{rem}
		Note that $\delta(\mathfrak{k}^{\bar{\Q}}, \mathcal{H}) \leq \dim_{\bar{\Q}} \mathfrak{k}^{\bar{\Q}} = m$, since one can take the trivial ideal $\{0\}$. Furthermore, suppose a non-zero vector $v\in \mathfrak{k}^{\bar{\Q}}$ is given. By definition of $\delta(\mathfrak{k}^{\bar{\Q}}, \mathcal{H})$, we can find an $\mathcal{H}$-invariant ideal $I^{\bar{\Q}}$, such that $v\notin I^{\bar{\Q}}$ and $\dim_{\bar{\Q}}(\mathfrak{k}^{\bar{\Q}}/I^{\bar{\Q}}) \leq \delta(\mathfrak{k}^{\bar{\Q}}, \mathcal{H})$.

		It should be noted that, although stated as a value depending on $\mathfrak{k}^{\bar{\Q}}$, it in fact an invariant of the complex Lie algebra $\mathfrak{k}^\C$ by the Lefschetz' Principle, see e.g. \cite[Chapter 3]{MR3967743}.
	\end{rem}
	
	In order to relate this constant to finite quotients, we need to introduce some notation. As before, we write $\mathcal{O}_\F$ for the ring of algebraic integers of a number field $\F$. Recall that $\F$ is the field of fractions of $\mathcal{O}_\F$. 
	
	\begin{df}
		Let $R$ denote a ring. If $S$ is a multiplicatively closed subset of $R$ with $1\in S$, then we denote $S^{-1}R$ for the \textbf{localization} of $R$ with respect to $S$. Recall that $S^{-1}R$ is given by the formal fractions $\{r/s \mid r\in R, s\in S\}$. 
	\end{df} 
	\noindent In our case, we will work with localizations of the form $S^{-1}\mathcal{O}_\F$ where $S=\{1, x,x^2, \ldots\}$ for some $x\in \mathcal{O}_\F$. Note that $S^{-1}\mathcal{O}_\F$ is equal to $\mathcal{O}_\F[1/x]$. 

	\begin{prop} \label{prop_delta_Zp}
		Let $\pi_\delta(r)$ denote the number of prime numbers $\max\{M, \Delta\} < p\leq r$ such that
		$$\min\{\max_{i=1}^k\{\dim_{\Z_p}(L^{\Z_p}/J_i)\}\mid J_1 \text{ to } J_k \text{ are }\mathcal{H}\text{-invariant ideals of }L^{\Z_p}, \;\cap_{i=1}^k J_i = \{0\}\} \leq \delta(L^{\bar{\Q}}, \mathcal{H}),$$
		and all $\xi \in \mathcal{H}$ have a Jordan Normal Form over $\Z_p$ preserving diagonalizability. Then, $\pi_\delta(r) \asymp r/\log(r)$, i.e. there exist constants $C_1,C_2 > 0$ such that (for all $r$ sufficiently large)
		$$C_1r/\log(r) \leq \pi_\delta(r) \leq C_2r/\log(r).$$
	\end{prop}
	\begin{proof}
		Take ideals $I_1^{\bar{\Q}}$ to $I_k^{\bar{\Q}}$ realizing the definition of $\delta(L^{\bar{\Q}}, \mathcal{H})$. In other words, ideals such that $\cap_{i=1}^k I_i^{\bar{\Q}} = \{0\}$ and $\dim_{\bar{\Q}}(L^{\bar{\Q}}/I_i^{\bar{\Q}}) \leq \delta(L^{\bar{\Q}}, \mathcal{H})$. We will now construct some auxiliary matrices. For this, identify $L^{\bar{\Q}}$ with coordinate vectors in $\bar{\Q}^m$ with respect to a basis of $L$:
		\begin{itemize}
			\item Define matrices $B^{(i)} \in \GL(m, \bar{\Q})$ for every $I_i^{\bar{\Q}}$ by letting the first columns represent a basis of $I_i^{\bar{\Q}}$ and extending it in the other columns to a basis of $L^{\bar{\Q}}$.
			\item Now, any element in $L^{\bar{\Q}}$ can be expressed as $B^{(i)}\lambda$ for some $\lambda \in \bar{\Q}^m$. This way, we define the vectors $\lambda^{(i)}_{k,l}$ and $\mu^{\xi}_{i,k}$ such that $[v_k,v_l]_L = B^{(i)}\lambda^{(i)}_{k,l}$ and $\xi(v_k) = B^{(i)}\mu^{\xi}_{i,k}$. Here, $\xi \in \mathcal{H}$, and $v_k$ and $v_l$ are the $k$'th and $l$'th column of $B^{(i)}$ respectively.
			\item Define a matrix $D$ as a block matrix with $k\times m$ blocks. The $(i,j)$ block is given by the projection of the $j$'th standard vector $e_j$ on the space spanned by the columns of $B^{(i)}$ not corresponding to the ideal $I_i^{\bar{\Q}}$, i.e. the last entries of the vector $(B^{(i)})^{-1}e_j$. (Note that $D$ therefore has $m$ columns.) 
		\end{itemize}
		Note that the construction of $D$ can be done for any set of ideals, even if they do not intersect trivially. In this case, $D$ satisfies the following property: the intersection $\cap_{i=1}^kI_i = \{0\}$ if and only if $D$ has rank $m$. Indeed, suppose first that $D$ is not of rank $m$, then we can find a vector $\mu\neq 0$ such that $D\mu = 0$. In particular, for all $1\leq i \leq k$ we have $D_i\mu = 0$, where $D_i$ denotes the matrix consisting of the blocks on the $i$'th level. By the way we defined the blocks, this implies that the vector $\mu$ must lie in $I_i^{\bar{\Q}}$, and this for all $1\leq i \leq k$. Conversely, suppose that $0\neq \mu \in \cap_{i=1}^k I_i$, then $D_i\mu = 0$ for all $1\leq i \leq k$, and hence, $D\mu = 0$. Therefore, $D$ cannot have rank $m$.
		
		Note that we have only constructed finitely many matrices and vectors. Hence, all entries surely lie over some number field $\F$. Moreover, we may suppose that $\F$ is Galois over $\Q$ and that the characteristic polynomials of all matrices corresponding to the automorphisms $\xi\in\mathcal{H}$ splits in this number field.
		
		Now, since the quotient field of $\mathcal{O}_\F$ is precisely $\F$, we can take $M_\F$ to be a common denominator of the eigenvalues of the matrices $\xi\in\mathcal{H}$ and of all the entries in the matrices and vectors defined above, but also a multiple of $\Delta$. 
		Define the ring $R$ to be $S^{-1}\mathcal{O}_\F$ with $S = \{1, M_\F, M_\F^2, \ldots\}$. Note that $\Z[1/\Delta] \subset R$.
		
		By Chebotarev's density theorem, more specifically Proposition \ref{prop_apply_cheby}, the number of primes smaller than $r$ such that there exists a ring homomorphism $\rho: \mathcal{O}_\F \to \Z_p$ has density $\pi_{\mathcal{O}_\F \to \Z_p}(r) \asymp r/\log(r)$. As it is noted below Proposition \ref{prop_apply_cheby}, we may exclude primes with corresponding ring homomorphisms $\rho$ such that $\rho(M_\F) = 0$, since this only excludes finitely many primes and thus does not affect the density result. Similarly, we restrict our attention to homomorphisms such that $\rho(b^{(i)}) \neq 0$, where $b^{(i)}\in \mathcal{O}_\F$ is the nominator of $\det B^{(i)}$. Also, since $D$ has rank $m$, we can take an $m\times m$ submatrix with non-zero determinant $D'$. Henceforth, we will also assume that $\rho(d') \neq 0$, where $d'\in \mathcal{O}_\F$ is the nominator of $D'$. Finally, if $\nu_1/M_\F^l\neq \nu_2/M_\F^l$ are distinct eigenvalues of $\xi\in\mathcal{H}$, we will suppose that $\rho(\nu_1) \neq \rho(\nu_2)$.
		In particular, since $\rho(M_\F) \neq 0$ and therefore invertible, the universal property of localization says our ring homomorphism $\rho$ extends to a ring homomorphism $\rho: R \to \Z_p$ (by sending $1/M_\F$ to $\rho(M_\F)^{-1}$).
		
		We can apply $\rho$ coordinate-wise (with respect to the basis of $L$) to obtain maps $\rho_L: L^R \to L^{\Z_p}$. The basis represented by the matrices $B^{(i)}$ are mapped to a basis of $L^{\Z_p}$ since $\rho(\det B^{(i)}) \neq 0$. The first vectors spanning the ideal $I_i^{\bar{\Q}}$ are mapped to vectors spanning a vector space $I_i^{\Z_p}$ in $L^{\Z_p}$. Applying $\rho_L$ to the definitions of $\lambda_{k,l}^{(i)}$ and $\mu^{\xi}_{i,k}$, we have
		$$[\rho_L(v_k), \rho_L(v_l)]_L= \rho_L([v_k,v_l]_L) = \rho_L(B^{(i)})\rho_L(\lambda_{k,l}^{(i)})$$
		and
		$$\bar{\xi}(\rho_L(v_k)) = (\rho_L\circ\xi\circ\rho_L^{-1})(\rho_L(v_k)) = \rho_L(\xi(v_k)) = \rho_L(B^{(i)})\rho_L(\mu^{\xi}_{i,k}).$$
		Therefore, $I_i^{\Z_p}$ must still be an $\mathcal{H}$-invariant ideal of $L^{\Z_p}$ (of the same dimension). Also, we made sure that $\rho_L(D)$ still has rank $m$. Hence, $\cap_{i=1}^n I^{\Z_p}_i = \{0\}$.
		Lastly, since the characteristic polynomials of the automorphisms $\xi \in \mathcal{H}$ split over $R$, they split over $\Z_p$, so $\bar{\xi}$ has a Jordan Normal Form. In fact, since distinct eigenvalues of $\xi$ are mapped to distinct eigenvalues of $\bar{\xi}$, the automorphism $\bar{\xi}$ is diagonalizable if $\xi$ is by considering the minimal polynomial.
	\end{proof}

	By \cite[Proposition 4.7]{math_virt_ab}, this density result has the following implication:
	\begin{cor} \label{cor_take_prime}
		There exists a constant $C_{\text{density}}>0$ such that, given a number $0\neq x\in \Z$, a prime $p \leq C_{\text{density}}\log(|x|)+C_{\text{density}}$ satisfying Proposition \ref{prop_delta_Zp} exists with $p\nmid x$.
	\end{cor}
	The extra condition on the primes, namely that the characteristic polynomial of all $\xi \in \mathcal{H}$ splits over $\Z_p$, was added to apply the following result.
	\begin{lemma} \label{lem_diagonal}
		Let $\bar{\xi} : L^{\Z_p} \to L^{\Z_p}$ be an isomorphism such that the characteristic polynomial splits over $\Z_p$, where $p>m$. Then, the order of $\bar{\xi}$ divides $(p-1)p$. If $\bar{\xi}$ is diagonalizable, then its order divides $p-1$.
	\end{lemma}
	\begin{proof}
		By the assumption, there exists a basis of $L^{\Z_p}$ such that the matrix corresponding to $\bar{\xi}$ is a Jordan Normal Form $J = D+N$, where $D \in \Z_p^{m\times m}$ is diagonal and $N$ has its non-zero values on the first off-diagonal. Now, the matrix corresponding to $\bar{\xi}^{k}$ is given by
		$$J^k = (D+N)^k = \sum_{i=0}^k\binom{k}{i}D^{k-i}N^{i} = D^k + \binom{k}{1}D^{k-1}N + \ldots + \binom{k}{m-1}D^{k-m+1}N^{m-1}.$$
		If $k$ is a multiple of $p-1$, then Fermat's little theorem states that $D^k \equiv \mathbb{1} \mod p$, yielding the diagonalizable case. Now, suppose that $k$ is also a multiple of $p$, then $p \mid \binom{k}{i}$ for all $1\leq i \leq m-1$, and therefore, $J^k \equiv \mathbb{1} \mod p$.
	\end{proof}
	Now, we will proceed to give upper bounds for $\RF_G$, as given by Theorem \ref{thm_intro_upper} from the Introduction. Let us first make some calculations:
	\begin{lemma} \label{lem_comm_obs}
	Take notations as in Section \ref{sec_nota}. Let $h\in \{h_j^{\pm 1}\mid 1\leq j \leq m\}$ with corresponding action $\xi$ on $K$. If the induced map by $\xi$ on the quotient $K/K^p$ has order dividing $o$, then
		\begin{itemize}
			\item $(hk)^{po}K^p = h^{po}K^p \in G/K^p$ for all $k\in K$,
			\item $[h^{po}, g] \in K^p$ for all $g \in \bar{G}$.
		\end{itemize}
		Furthermore, if $\bar{G} = K \rtimes_\varphi \Z^n$ for a morphism $\varphi: \Z^n \to \text{Aut}(K)$, then $[h^{o}, g] \in K^p$ for all $g \in \bar{G}$.
	\end{lemma}
	\begin{proof}
		Let $\pi: G \to H$ denote the projection onto the virtually abelian group. For the first observation, we have
		$$\pi\left((hk)^{o}\right) = \pi\left(h^{o}\right).$$
		Hence, there exists some $\tilde{k}\in K$ such that $(hk)^{o} = h^{o}\tilde{k}$. Now, we find
		$$(hk)^{po}K^p =  \left(h^{o}\tilde{k}\right)^pK^p = h^{po}\tilde{k}^pK^p,$$
		where we used that $\tilde{k}h^{o}K^p = h^{o}\xi^{o}(\tilde{k})K^p = h^{o}\tilde{k}K^p$ by the assumption.
		
		For the second observation, we have
		$$[h^{po}, g] = h^{-po}\left(g^{-1}hg\right)^{po} .$$
		We know that $\pi(g^{-1}hg) = \pi(h)$ since $\bar{G}$ maps to $\Z^n$, so $g^{-1}hg = h\tilde{k}$ for some $\tilde{k}\in K$. Now, we conclude by using the first observation:
		$$[h^{po}, g]K^p = h^{-po}\left(h\tilde{k}\right)^{po}K^p = h^{-po}h^{po}K^p = K^p. $$
		
		For the `furthermore' part, recall that $h_j$ is given by $(e,e_j) \in K\rtimes_\varphi \Z^n$, so we can write $h$ as $(e,h)$. It also follows that $\varphi(h) = \xi$. Since $\varphi(h^o)$ is the identity homomorphism on $K/K^p$, the element $h^0$ commutes with elements in $K$ and hence the statement easily follows. 
	\end{proof}

	Now, we prove a more detailed version of Theorem \ref{thm_intro_upper}:
	\begin{thm} \label{thm_upper}
	Using the notation introduced in Notations \ref{nota_group}-\ref{nota_poly} and Notations \ref{nota_Lie}-\ref{nota_LR_groups}, 
	we have
	$$\RF_G \preceq [r\mapsto r^{\delta(L^{\bar{\Q}}, \mathcal{H})+(1+\epsilon_1+\epsilon_2+\epsilon_3)n}]_\sim \preceq [r\mapsto r^{m + 4n}]_\sim.$$
	Here,
	\begin{itemize}
		\item $\epsilon_1 = 0$ if $H = \Z^n$, and $\epsilon_1 = 1$ otherwise;
		\item $\epsilon_2 = 0$ if $\bar{G} = K\rtimes_\varphi \Z^n$ for some $\varphi: \Z^n \to \Aut(K)$, and $\epsilon_2 = 1$ otherwise;
		\item $\epsilon_3 = 0$ if all homomorphisms in $\{\xi_j\mid 1\leq j \leq n\}$ are diagonalizable (over $\bar{\Q}$ or $\C$), and $\epsilon_3 = 1$ otherwise.
	\end{itemize}
	
	Furthermore, if $H$ is finite, then
	$$\RF_G \preceq [r\mapsto \log^{\delta(L^{\bar{\Q}}, \mathcal{H})}(r)]_\sim. $$
	\end{thm}
	\begin{proof}
		Take a non-trivial element $g\in B_G(r)$. Recall that $H$ is residually finite with a finite-index, free abelian subgroup of rank $n$. If $\pi(g) \neq e$, then the residual finiteness growth of the virtually abelian group $H$ dictates that
		$$D_G(g) \leq D_H(\pi(g)) \preceq \log^n(r)$$
	by \cite[Theorem 1.2]{math_virt_ab}. Note that this function grows slower than the upper bound we wish to demonstrate. Therefore, we may henceforth assume that $e\neq g\in B_G(r)\cap K$.
		
		By Proposition \ref{prop_Br_to_basis}, we know that $g$ can be written as 
		$$g = \lambda_1v_1 + \ldots + \lambda_m v_m$$
		with respect to a chosen fixed $\Z$-basis $\{v_1, \ldots , v_m\}$. The coefficients $\lambda_i$ lie in $\Z[1/\Delta]$, are of the form $\lambda_i = \mu_i/\Delta^{j_i}$ with $\mu_i\in\Z$ and $j_i\in\N$, where $|\mu_i| \leq C^r$ for some fixed constant $C>0$. If $|H| < \infty$, then $|\mu_i| \leq Cr^C$. In both cases, denote $b(r)$ for this upper bound.
		
		Since $g$ is non-trivial, one of the $\lambda_i$ is non-zero, say $\lambda_{i_0}$. Let $M$ denote the bound given in Lemma \ref{lem_sega_normal_ideal}. Now, take a prime $p> \max\{M, \Delta, m\}$ such that $p \nmid \mu_{i_0}$. By Corollary \ref{cor_take_prime}, we may assume that $p$ satisfies Proposition \ref{prop_delta_Zp} and $p\leq C_{\text{density}}\log(b(r)) + C_{\text{density}}$ for some fixed constant $C_{\text{density}}$. By construction, $g\in L^\Delta \setminus pL^\Delta$.
		
		Lemma \ref{lem_LDelta_LZp} states that $L^\Delta/pL^\Delta \cong L^{\Z_p}$. By Proposition \ref{prop_delta_Zp}, we obtain $\mathcal{H}$-invariant ideals $J_1$ to $J_k$ with $\cap_{i=1}^k J_i = \{0\}$ and $|L^{\Z_p}/J_i| \leq \delta(L^{\bar{\Q}}, \mathcal{H})$. Since they have trivial intersection, we can take one of them, denoted by $J$, such that $g \notin J$. Now, lemma \ref{lem_LDelta_LZp} guarantees that the preimage of $J$ in $L^\Delta$, $\psi^{-1}(J) \subset L^\Delta$, is an $\mathcal{H}$-invariant ideal of the same index. Denote $\psi^{-1}(J)\cap K$ by $N_1$. By Lemma \ref{lem_normal_ideal}, we know that $N_1$ is an $\mathcal{H}$-invariant normal subgroup of $K$ with $[K:N_1] \leq \delta(L^{\bar{\Q}}, \mathcal{H})$. Note that $pL^\Delta \cap K \subset N_1$ by construction, and $K^p = pK \subset pL^\Delta \cap K$ by \cite[2.2.5]{robinson} as $p > m$.
		
		Recall we have the following short exact sequence:
		$$1 \to K \to G \to H \to 1.$$
		Here, $G$ is generated by $S = \{k_i,h_j,f_s\mid 1\leq i \leq m, 1\leq j \leq n, 1\leq s \leq [H:\Z^n]-1\}$. Since $N_1$ is normal in $K$, and  $s^{-1}N_1s = N_1$ for all $s\in S$ by $\mathcal{H}$-invariance, we conclude that $N_1$ is normal in $G$ itself. Hence, we can define $\varphi_1: G \to G/N_1$, satisfying $\varphi_1(g) \neq e$. This results in the short exact sequence of the form
		$$ 1 \to K/N_1 \to G/N_1 \to H \to 1.$$
		
		Define $\epsilon_1$, $\epsilon_2$ and $\epsilon_3$ as in the statement of the theorem. We claim that 
		$$N_2 = \langle N_1, h_j^{(p-1)p^{(\epsilon_1 + \epsilon_2 + \epsilon_3)}} \mid 1\leq j \leq n\rangle$$
		is a normal subgroup of $G$ with $K\cap N_2 = N_1$, and in particular, $g\notin N_2$.
		
As we already argued that $s^{-1}N_1s \in N_2$ for all $s\in S$, we will proceed to show that $$s^{-1}h_j^{(p-1)p^{(\epsilon_1 + \epsilon_2 + \epsilon_3)}}s \in N_2$$ or equivalently $[s,h_j^{(p-1)p^{(\epsilon_1 + \epsilon_2 + \epsilon_3)}}] \in N_2$ for every $s\in S$ and $1\leq j \leq n$ to show that $N_2$ is a normal subgroup. By the choice of $p$ as in Proposition \ref{prop_delta_Zp}, we know that the characteristic polynomials of all $\{\xi_j\mid 1\leq j \leq n\}$ splits over $\Z_p$. By Lemma \ref{lem_diagonal}, we therefore know that their order (over $\Z_p$) divides $(p-1)p^{\epsilon_3}$. Using this in Lemma \ref{lem_comm_obs} gives us the observation that 
		\begin{equation} \label{eq_hj_commutes}
		[h_j^{(p-1)p^{\epsilon_3}p^{\epsilon_2}}, g] \in K^p \leq N_1 \leq N_2
		\end{equation}
		for all $g \in \{k_i, h_j\mid 1\leq i \leq m, 1 \leq j \leq n\}$ and $1\leq j \leq m$. The same lemma also guarantees that $[h_j^{(p-1)p^{\epsilon_1+\epsilon_2+\epsilon_3}}, g] \in N_2$. In particular, we have already shown that $N_2$ is normal in $\bar{G}$. 
		
		Take an element of the form $f_s$, so in particular we have $\epsilon_1 = 1$. It suffices to show that $f_s^{-1}h_j^{(p-1)p^{\epsilon_2+\epsilon_3+1}}f_s \in N_2$.  
		Note that by construction $f_s^{-1}\bar{G} f_s = \bar{G}$, and hence $f_s$ induces an action on $\Z^n$ via conjugation. 

		Consider $\pi(h_j^{(p-1)p^{\epsilon_2+\epsilon_3}})$ with $1\leq j \leq n$. Recall that $\{\pi(h_j)\mid 1\leq j \leq n\}$ is a basis of $\Z^n$, and thus 
		$$\pi(f_s^{-1}h_j^{(p-1)p^{\epsilon_2+\epsilon_3}}f_s) = \prod_{l=1}^n \pi(h_l)^{d_l(p-1)p^{\epsilon_2+\epsilon_3}} = \pi\left(\prod_{l=1}^n h_l^{d_l(p-1)p^{\epsilon_2+\epsilon_3}}\right)$$
		for some $d_l \in \Z$ (and $1\leq l \leq n$). In particular,
		$$f_s^{-1}h_j^{(p-1)p^{\epsilon_2+\epsilon_3}}f_s = \left(\prod_{l=1}^n h_l^{d_l(p-1)p^{\epsilon_2+\epsilon_3}}\right)\tilde{k}$$
		for some $\tilde{k} \in K$. Now, in $G/N_1$, we have
		\begin{equation*}
			\begin{split}
				f_s^{-1}h_j^{(p-1)p^{\epsilon_2+\epsilon_3 + 1}}f_s N_1 & = \left(f_s^{-1}h_j^{(p-1)p^{\epsilon_2+\epsilon_3}}f_s\right)^p N_1 \\
				& = \left(\left(\prod_{l=1}^n h_l^{d_l(p-1)p^{\epsilon_2+\epsilon_3}}\right)\tilde{k}\right)^p N_1 \\
				& = \left(\prod_{l=1}^n h_l^{d_l(p-1)p^{\epsilon_2+\epsilon_3+1}}\right)\tilde{k}^p N_1 \\
				& = \left(\prod_{l=1}^n h_l^{d_l(p-1)p^{\epsilon_2+\epsilon_3+1}}\right) N_1,
			\end{split}
		\end{equation*}
		where we used that $h_l^{(p-1)p^{\epsilon_2+\epsilon_3}}N_1$ is central in $\bar{G}/N_1$ by Equation \eqref{eq_hj_commutes}. We conclude that $f_s^{-1}h_j^{(p-1)p^{\epsilon_2+\epsilon_3 + 1}}f_s \in N_2$, and thus $N_2 \lhd G$.
		
		Now, we will argue that $K\cap N_2 = N_1$. Therefore, suppose $\tilde{g} \in K\cap N_2$. By definition of $N_2$, $\tilde{g}$ can be written as a product of elements in $N_1$ and elements of the form $h_j^{\pm(p-1)p^{\epsilon_1+\epsilon_2+\epsilon_3}}$. Note that $\pi(\tilde{g}) = 0$. Since $\{\pi(h_j)\mid 1\leq j \leq n\}$ is a basis of $\Z^n$, the number of elements $h_l^{(p-1)p^{\epsilon_1+\epsilon_2+\epsilon_3}}$ and $h_l^{-(p-1)p^{\epsilon_1+\epsilon_2+\epsilon_3}}$ in this product must be the same. Since these elements are central modulo $N_1$ by Equation \eqref{eq_hj_commutes}, we can rewrite this product such that they cancel. We are left with an element in $N_1$. This shows the claim.
		
		To finish the proof, we note that by construction $g\notin N_2$, and hence
		\begin{equation*}
			\begin{split}
				D_G(g) & \leq [G:N_2] = [K:K\cap N_2]\cdot [H: \pi(N_2)] \\
				& = [K : N_1] \cdot [H: \Z^n] \cdot [\Z^n: \pi(N_2)] \\
				& = [L^{\Z_p}: J] \cdot [H: \Z^n] \cdot [\Z^n : (p-1)p^{\epsilon_1+\epsilon_2+\epsilon_3}\Z^n] \\
				& \leq p^{\delta(L^{\bar{\Q}}, \mathcal{H})}\cdot [H: \Z^n] \cdot p^{(1+\epsilon_1+\epsilon_2+\epsilon_3)n} \\
				&\leq [H: \Z^n]\cdot \left(C_{\text{density}}\log(b(r))+C_{\text{density}}\right)^{\delta(L^{\bar{\Q}}, \mathcal{H}) + (1+\epsilon_1+\epsilon_2+\epsilon_3)n}.
			\end{split}
		\end{equation*}
	If $H$ is finite, then $b(r)$ was polynomial. By the property that $\log(r^d) = d\log(r)$, we conclude that $D_G(g) \preceq \log^{\delta(L^{\bar{\Q}}, \mathcal{H})}(r)$ as required. If $H$ is infinite, then $b(r) \leq C^r$ and hence $\log(b(r)) \preceq r$, yielding the bound given in the theorem's statement. Note that $\delta(L^{\bar{\Q}}, \mathcal{H}) \leq m = r(K^\Q)$.
	\end{proof}
	\begin{ex}
		Let $G$ be \textbf{virtually polycyclic}, so $G$ has a normal series where the quotients are either cyclic or finite. 
The number of infinite cyclic factors is called the Hirsch length of $G$, denoted by $h(G)$, and is a group invariant by \cite[p. 16]{sega83-1}. In fact, $h(G) = h(K) + h(G/K) = m + n$, and therefore, $\RF_G \preceq r^{4h(G)}$. 
	\end{ex}
	\begin{ex}\label{ex_Z2Z}
		Consider the group $G = \Z^2\rtimes_\varphi \Z$ with $\varphi(1): \Z^2 \to \Z^2: v \mapsto Av$, where $A = \begin{psmallmatrix}2&1\\1&1\end{psmallmatrix}$. The action of $\varphi(1)$ on $\Z^2$ is diagonalizable, and thus the eigenspaces over $\bar{\Q}^2$ yield invariant ideals with trivial intersection. Therefore, $\delta(L^{\bar{\Q}}, \mathcal{H}) = 1$ and hence $\RF_G \preceq r^2$ via Theorem \ref{thm_upper}. We find the same upper bound for the Baumslag-Solitar groups $\BS(1,n) = \Z[1/n]\rtimes \Z$, since $r(\Z[1/n]) = m = 1$.
	\end{ex}
	\begin{rem}
		\label{ex:motivation}
		In \cite{franz2017quantifying}, an upper bound $\RF_G \preceq r^{l^2-1}$ for linear groups $G \leq \GL(l, \C)$ was communicated. This bound is quadratic in the `dimension of linearity' $l$. In contrast, Theorem \ref{thm_upper} gives a bound that is linear in the rank of the $\mathcal{M}$-group $G$. This might yield sharper bounds for possibly a large class of groups. For example, in $G = \Z^2\rtimes_\varphi \Z$ as above, the linear embedding
		$$G \hookrightarrow \GL(3, \Z): (v,l)\mapsto \begin{pmatrix} A & v \\ 0 & 1\end{pmatrix}$$
		provides a bound $\RF_G \preceq r^8$, while Theorem \ref{thm_upper} says that $\RF_G \preceq r^2$. 	In general, the difference between these bounds can become arbitrary large. 
		
	Note that for a general finitely generated minimax group $G$, one expects that the minimal $l$ such that $G$ can be realized as a subgroup of $\GL(l,\C)$ is at least the Hirsch lenght $h(G)$. For example, this is the case for finitely generated nilpotent groups associated to a filiform nilpotent algebra $\mathfrak{g}$ see \cite[Proposition 2]{MR1411303}.	Moreover, the minimal $l$ that allows an embedding $G \hookrightarrow \GL(l, \C)$ is influenced by the finite group $H/\Z^n$ of the $\mathcal{M}$-group, while the bound $\RF_G \preceq r^{m+4n}$ of Theorem \ref{thm_upper} is not.
	\end{rem}
	\begin{ex} \label{ex_virt_ab}
	Let $G$ be \textbf{virtually abelian} with free abelian subgroup $K = \Z^m$ of maximal rank and finite $H = G/K$. In this case, Theorem \ref{thm_upper} says that
	$$\RF_G \preceq \log^{\delta(L^{\bar{\Q}}, \mathcal{H})}.$$
	Since $K$ is abelian, its corresponding Lie algebra is $K^\Q$ itself (with trivial Lie bracket). In particular, $L^{\bar{\Q}} = \bar{\Q}^m$. By consequence, $\delta(L^{\bar{\Q}}, \mathcal{H})$ equals
	\begin{equation*}
		\delta(\bar{\Q}^m, \mathcal{H}) = \min\{\max_{i=1}^k\{\dim_{\bar{\Q}}(\bar{\Q}^m/I_i^{\bar{\Q}})\}\mid I_1^{\bar{\Q}} \text{ to } I_k^{\bar{\Q}} \text{ are }\mathcal{H}\text{-invariant subspaces of }\bar{\Q}^m, \;\cap_{i=1}^k I_i^{\bar{\Q}} = \{0\}\}.
	\end{equation*}
	If we decompose $\bar{\Q}^m$ into a direct sum of absolutely irreducible subspaces, $\bar{\Q}^m = V_1^{\bar{\Q}} \oplus \ldots \oplus V_k^{\bar{\Q}}$, then the $I_i^{\bar{\Q}}$ realizing $\delta(\bar{\Q}^m, \mathcal{H})$ equal $I_i^{\bar{\Q}} = V_1^{\bar{\Q}} \oplus \ldots \oplus V_{i-1}^{\bar{\Q}} \oplus V_{i+1}^{\bar{\Q}}\oplus \ldots \oplus V_k^{\bar{\Q}}$. Hence, $\delta(\bar{\Q}^m, \mathcal{H})$ equals the largest dimension of an absolutely irreducible subspace. This is precisely the bound communicated in \cite{math_virt_ab}. In fact, that paper shows the bound is exact for these groups.
	\end{ex}
	\subsection{Remarks on exactness}
In this subsection, we discuss two main obstructions to the exactness of the upper bound given in Theorem \ref{thm_upper}. The first one is due to the fact that the quotient in $H$ might be larger than needed. The second one is that $G$ might split with a nilpotent part. Examples \ref{ex_Z2Z_2} and \ref{ex_splitting_semidirect} will illustrate these obstructions.
\begin{ex}\label{ex_Z2Z_2}
	Take the group $G = \Z^2\rtimes_\varphi \Z$ of Example \ref{ex_Z2Z}. Now consider $G\times G$. From one side, we have $\RF_{G\times G} = \max\{\RF_G, \RF_G\} = \RF_G \preceq r^2$. From the other side, Theorem \ref{thm_upper} yields $\RF_{G\times G} \preceq r^3$ (for the same reasons as in Example \ref{ex_Z2Z}). From this we conclude that the bound in Theorem \ref{thm_upper} is not sharp in general. A major obstruction to exactness is the choice that $\pi(N_2) = l\Z^2 \lhd H$ for some $l\in \N$ in the proof of the theorem. From \cite{math_virt_ab}, we know that normal subgroups realizing $\RF_H$ are, in most cases, not of this form. Therefore, we should expect that setting $N_2 = \langle N_1, h_j^l\mid 1\leq j \leq n\rangle$ is not optimal. 
\end{ex}
In this example it holds that $n=2$, because $H = \Z^2$. However, using other properties, we can reduce the problem to a case where $n$ is smaller. This automatically decreases the estimate of Theorem \ref{thm_upper}. The following observation is important in this setting:
\begin{prop}\label{prop_RF_covering}
	Let $G$ be a finitely generated group. Let $\{\pi_i: G \to G_i\mid 1\leq i \leq n\}$ denote a finite set of surjective homomorphisms from $G$ to residually finite groups $G_i$ ($1\leq i \leq n$). If $\cap_{i=1}^n\ker\pi_i = \{e\}$, then $G$ is residually finite and $\RF_G \preceq \max\{\RF_{G_i}\mid 1\leq i\leq n\}$. 
\end{prop}
\begin{proof}
	Let $S$ be finite generating set of $G$. Now, $\pi_i(S)$ is a finite generating set of $G_i$ and $\pi(B_{G,S}(r)) = B_{G, \pi_i(S)}(r)$. Take $e\neq g\in B_{G,S}(r)$ arbitrary. Since $g\notin \cap_{i=1}^n\ker\pi_i$, there is some index $1\leq j \leq n$ such that $\pi_j(g) \neq e$. By the residual finiteness growth of $G_j$, we know that there exists a homomorphism $\varphi: G_j \to Q$ with $\varphi(\pi_j(g)) \neq e$ and $|Q| \leq \RF_{G_j, \pi_j(S)}(r)$. Since $(\varphi\circ \pi_j)(g) \neq e$, we observe that
	$$D_{G}(g) \leq |Q| \leq \RF_{G_j, \pi_j(S)}(r) \leq \max\{\RF_{G_i, \pi_i(S)}(r)\mid 1\leq i\leq n\}.$$
	The result follows by taking the maximum over all non-trivial elements in $g\in B_{G,S}(r)$.
\end{proof}
In particular, if we have a short exact sequence of the form
$$1 \to K \to G \to H_1\times H_2 \to 1,$$
we can use the groups $G_1 = K$, $G_2 = \pi^{-1}(H_1)$ and $G_3 = \pi^{-1}(H_2)$ with $\pi$ the projection $G \to H_1\times H_2$. Hence, we obtain the bound
$$\RF_G \preceq \max\{\RF_K, \RF_{\pi^{-1}(H_1)},\RF_{\pi^{-1}(H_2)}\}.$$
Here, the three values $n$ are smaller than (or equal to) the original value $n$ for $G$, which yields improved estimates via Theorem \ref{thm_upper}. In particular, this result can always be applied when $H= \Z^n$. In this case, it reduces to groups where $n=1$. Furthermore, extensions by $\Z$ are always semidirect products, so $\epsilon_2$ becomes $0$.

Another obstruction to exactness will be discussed in Example \ref{ex_splitting_semidirect}. It is linked to the action of $H$ on $K$. Let us first prove a corollary of Proposition \ref{prop_RF_covering}:
\begin{lemma}
	Let $G_1\rtimes_{\varphi_1} G_3$ and $G_2\rtimes_{\varphi_2} G_3$ be two finitely generated residually finite groups. The group $G = (G_1\times G_2)\rtimes_{\varphi_1\times \varphi_2} G_3$ has its residual finiteness growth given by
	$$\RF_G = \max\{\RF_{G_1\rtimes_{\varphi_1} G_3}, \RF_{G_2\rtimes_{\varphi_2} G_3}\}.$$
\end{lemma}
\begin{proof}
	This result follows directly from the observations that $G_1\rtimes_{\varphi_1} G_3$ and $G_2\rtimes_{\varphi_2} G_3$ inject into $G$, and
	$$\psi_1: G \to G/G_{2} \cong G_1\rtimes_{\varphi_1} G_3  \text{ and }\psi_2: G \to G/G_{1} \cong G_2\rtimes_{\varphi_2} G_3$$
	are will defined maps with trivially intersecting kernels.
\end{proof}
\begin{ex} \label{ex_splitting_semidirect}
	Let $\Z$ act on $\Z^2$ via the matrix $\begin{psmallmatrix}2&1\\1&1\end{psmallmatrix}$, and let it act trivially on $H_3(\Z)$. The group $G = (\Z^2\times H_3(\Z)) \rtimes \Z$ defined in this way satisfies 
	$$\RF_G = \max\{\RF_{\Z^2\rtimes \Z}, \RF_{H_3(\Z)\times \Z}\} \preceq \max\{r^2, \log^3(r)\} \preceq r^2,$$
	using the lemma above and Example \ref{ex_Z2Z}. 
	
	According to the upper bound given in Theorem \ref{thm_upper}, we would have obtained the bound $\RF_G \preceq r^{3+1}$, since $H_3(\Z) \subset K$. However, the effect of the higher nilpotency class of $K$ can be estimated with a polylogarithmic bound in this example.
\end{ex}
Note that both obstructions require that $H$ is infinite. Since we already know that our bound is exact for all virtually abelian groups (see Example \ref{ex_virt_ab}), we conjecture this to be the case for all virtually nilpotent groups:
\begin{conj}
	Using Notations \ref{nota_Lie}-\ref{nota_LR_groups}, if $G$ is finitely generated virtually nilpotent with torsion-free nilpotent normal subgroup $K$ and finite quotient $H = G/K$, then
	$$\RF_G = [r\mapsto \log^{\delta(L^{\bar{\Q}}, \mathcal{H})}(r)]_\sim. $$
\end{conj}
\begin{rem}
	The bound $\RF_G \preceq \log^{\delta(L^{\bar{\Q}}, \mathcal{H})}$ depends on the complex Mal'cev completion of $K$ and the induced action of $H$ on it. In \cite[Question 3]{math_virt_ab}, the authors asked whether $$G_1^\C \cong G_2^\C \Rightarrow \RF_{G_1} = \RF_{G_2}$$ holds for finitely generated torsion-free nilpotent groups. This upper bound can be seen as a partial positive answer to this question. If the conjecture above would hold, then we would obtain a full positive answer.
\end{rem}

	\section{Lower Bound}\label{sec_lower}
	In Theorem \ref{thm_upper}, we have seen that finitely generated virtually nilpotent groups, so with $H$ finite, admit a polylogarithmic upper bound, while there is only a polynomial upper bound for the other groups. In this subsection, we illustrate that $r\preceq \RF_G$ for those remaining groups, as stated below. To prove this, we will make use of the exponential word growth of these groups. As mentioned in the introduction, this result gives a generalization of \cite[Theorem 1.1]{Mark_strict_dist}.
	\begin{thm} \label{thm_lower}
		Let $G$ be an $\mathcal{M}$-group, then,
		\begin{enumerate}[(i)]
			\item $G$ is virtually nilpotent if and only if $\RF_G \preceq \log^s$ for some $s\in \N$;
			\item $G$ is not virtually nilpotent if and only if $r \preceq \RF_G$.
		\end{enumerate}
	\end{thm} 
	We will use the following lemma:
	\begin{lemma} \label{lem_pigeon_Zn}
		Using the notation and the basis introduced in Notations \ref{nota_group}-\ref{nota_basis}, if there exists a constant $C_1>1$ such that $C_1^r \leq |B_{\bar{G}}(r)|$, then there exists a constant $C_2>1$ such that $C_2^r \leq |B_{\bar{G}}(r)\cap K|$.
	\end{lemma}
	\begin{proof}
		Let $\pi: \bar{G} \to \Z^n$ denote the natural projection. Recall that the generators $\{h_j\mid 1\leq j \leq n\}$ are mapped to the standard generators of $\Z^n$, the other generators $\{k_i\mid 1\leq i \leq m\}$ are mapped to the neutral element $0\in \Z^n$. Hence, it is clear that $\pi(B_{\bar{G}}(r)) \subset B_{\Z^n}(r)$. For every $v \in B_{\Z^n}(r)$ define the set
		$$S_v = \{g \in B_{\bar{G}}(r)\mid \pi(g) = v\}.$$
		In total, there are at most $|B_{\Z^n}(r)| \leq (2r+1)^n$ such sets that are non-empty. However, there are $C_1^r \leq B_{\bar{G}}(r)$ elements to be divided among them. Hence, by the pigeonhole principle, there is a vector $w\in B_{\Z^n}(r)$ such that $|S_w| \geq C_1^r/(2r+1)^n$.
		
		Suppose $w = \pi(h_1^{l_1}\cdots h_n^{l_n})$ with $|l_1|+\ldots + |l_n| \leq r$. Then, for every $g\in S_w$, we have $gh_1^{-l_1}\cdots h_n^{-l_n} \in K\cap B_{\bar{G}}(2r)$, and moreover if $g_1 \neq g_2$ in $S_w$, then $g_1h_1^{-l_1}\cdots h_n^{-l_n}\neq g_2h_1^{-l_1}\cdots h_n^{-l_n}$. Therefore,
		$$|B_{\bar{G}}(2r)\cap K| \geq |S_w| \geq \dfrac{C_1^r}{(2r+1)^n}.$$
		From this, we conclude that $C_2>1$ exists such that $C_2^r \leq |B_{\bar{G}}(r)\cap K|$ for large enough $r$. 
	\end{proof}
	\begin{proof}[Proof of Theorem \ref{thm_lower}]
		If $G$ is virtually nilpotent, then Theorem \ref{thm_upper} implies that $\RF_G \preceq \log^s$. It suffices to argue that if $G$ is not virtually nilpotent, then $r\preceq \RF_G$. So, we assume that $G$ is not virtually nilpotent, and thus $\bar{G} \lhd_f G$ is also not virtually nilpotent. We will show that $r\preceq \RF_{\bar{G}}$, and therefore also $r\preceq \RF_G$.
		
		According to \cite[Theorem 4.8]{Wolf} a finitely generated solvable group which is not virtually nilpotent, such as $\bar{G}$, has exponential word growth. By this, we mean that there exist constants $C_1,C_3>1$ such that
		$$C_1^r \leq |B_{\bar{G}}(r)| \leq C_3^r .$$
		
		Recall that $\bar{G}$ fits in a short exact sequence of the form 
		$$1\to K \to \bar{G} \to \Z^n \to 1.$$
		We claim that there exists a constant $C_4>0$ such that the ball $B_{\bar{G}}(C_4r)$ contains an element of the form $g^{\lcm(1,2, \ldots , r)}$ with $g\in K$. From this, the claimed result follows directly. Indeed, if $\varphi: \bar{G} \to Q$ is a homomorphism to a finite group such that $\varphi(g^{\lcm(1,2, \ldots , r)}) \neq e$, then $|Q| \geq r$, so
		$$ r \leq \max\{D_{\bar{G}}(g) \mid g\in B_{\bar{G}}(C_4r)\} = \RF_{\bar{G}}(C_4r).$$
		
		We have that $C_1^r \leq |B_{\bar{G}}(r)|$. Using the pigeonhole principle, Lemma \ref{lem_pigeon_Zn} implies the existence of a constant $C_2 >1$ such that $C_2^r \leq |B_{\bar{G}}(r)\cap K|$. Consider the function $f(r) = \lcm(1,2, \ldots , r)^c$ with $c$ the nilpotency class of $K$, and recall that $K^{f(r)}$ denotes the normal subgroup $\langle g^{f(r)}\mid g\in K\rangle$. Note that $|K/K^{f(r)}| \leq f(r)^m \leq C_5^r$ for some constant $C_5>1$, since $f(r)$ can be exponentially bounded by the Prime Number Theorem, see for example \cite[Proposition 2.1, p. 189]{stein2010complex}. Now take $C_6>1$ such that $C_7 := C_2^{C_6}$ is strictly greater than $C_5$. By this choice, we have 
		$$|K/K^{f(r)}| \leq C_5^r < C_7^r = C_2^{C_6r} \leq |B_{\bar{G}}(C_6r)\cap K|.$$
		Hence, by the pigeonhole principle, there must be two distinct elements $g_1$ and $g_2$ in $B_{\bar{G}}(C_6r)\cap K$ such that $g_1K^{f(r)} = g_2K^{f(r)}$. Now, $g_1^{-1}g_2$ is a non-trivial element of $K^{f(r)}\cap B_{\bar{G}}(2C_6r)$. By \cite[Chapter 6, Proposition 2]{sega83-1}, every element in $K^{f(r)} = K^{\lcm(1,2, \ldots , r)^c}$ is of the form $g^{\lcm(1,2, \ldots , r)}$. In particular, the element $g_1^{-1}g_2 \in B_{\bar{G}}(C_4r)$ is where we set $C_4 = 2C_6$.
	\end{proof}
	These bounds can be sharpened if one has information about the growth of $|B_G(r) \cap \gamma_l(K)|$, where $\gamma_l(K)$ denotes the $l$'th term of the lower central series of $K$. 

	\begin{thm} \label{prop_lower_addendum}
		Let $G$ be an $\mathcal{M}$-group with torsion-free nilpotent normal subgroup $K$, following Notation \ref{nota_group}. If there exist constants $C_1,C_2>1$ and an integer $l>1$ such that $C_1^r \leq |B_G(C_2r) \cap \gamma_l(K)|$ for all $r>0$ sufficiently large, then $r^{l+1} \preceq \RF_G$.
	\end{thm}
We first prove the following result about the order of finite nilpotent groups:
\begin{lemma} \label{lem_bound_p_group}
	Let $P$ be a $p$-group of nilpotency class $l+1$. If $\gamma_l(P)$ has exponent $p^k$, then $$\vert P \vert \geq p^{k(l+1)}.$$
\end{lemma}
\begin{proof}Let $F_i$ denote the abelian group $F_i = \gamma_i(P)/\gamma_{i+1}(P)$. By \cite[Theorem 1.2.11]{robinson}, the map $$\varphi_i: F_i\otimes_\Z F_1 \to F_{i+1}: y_1\gamma_{i+1}(P)\otimes y_2\gamma_2(P) \to [y_1,y_2]\gamma_{i+2}(P)$$ is a well-defined surjective morphism of abelian groups for every $1\leq i \leq l-1$. In particular, it implies that the exponent of $F_l$ divides the exponent of $F_i$ for every $1 \leq i \leq l-1$. As $$ |P| = |F_1| \cdot |F_2| \cdot \ldots \cdot |F_l|$$ and $\vert F_i \vert \geq p^k$ by the previous, it suffices to show that $\vert F_1 \vert \geq p^{2k}$. 
	
As $F_2$ has exponent at least $p^k$, there exists elements $x_1, x_2 \in F_1$ such that the element $y =\varphi_1(x_1 \otimes x_2)$ has order at least $p^k$, as the elements of this form generate $F_2$. Because $\varphi_1$ is a morphism, the elements $x_1$ and $x_2$ have order at least $p^k$. We claim that all element of the form $x_1^{i_1}x_2^{i_2}$ with $0\leq i_1, i_2 <p^k$ are distinct, which implies that $\vert F_1 \vert \geq p^{2k}$. Otherwise, by interchanging $x_1$ and $x_2$ if necessary, there exists an integer $1\leq j_1 < p^k$ such that $x_1^{j_1} = x_2^{j_2}$ for some $j_2 \in \Z$. In particular, $$y^{j_1} = \varphi_1(x_1 \otimes x_2)^{j_1} = \varphi_1(x_1^{j_1} \otimes x_2)  = \varphi_1(x_2^{j_2} \otimes x_2) =  0,$$ which contradicts that the order of $y$ is at least $p^k$.
\end{proof}
We now proceed to prove the theorem:
\begin{proof}[Proof of Theorem \ref{prop_lower_addendum}]
By the same argument as in the proof of Theorem \ref{thm_lower}, there exists a constant $C>0$ such that we can find a non-trivial element of the form $g^{\lcm(1,2, \ldots ,r)}$ in $B_G(Cr) \cap \gamma_l(K)$. Let $\varphi: G \to Q$ denote a homomorphism to a finite group such that $\varphi(g^{\lcm(1,2, \ldots ,r)}) \neq e$ and $|Q| = D_G(g^{\lcm(1,2, \ldots ,r)})$. We claim that $|Q| \geq r^{l+1}$, showing that $\RF_G(Cr) \geq r^{l+1}$.
	
	Since $g \in K$, we know that $\varphi(g) \in \varphi(K)$, which is nilpotent. As a finite nilpotent group is a direct sum of finite $p$-groups, we can compose the restriction of $\varphi$ to $K$ with a projection onto one of the $p$-groups $P$ to find a morphism $\psi :K \to P$ such that $\psi\left(g^{\lcm(1,2, \ldots ,r)}\right) \neq e$. Note in particular that $\psi(g) \in \gamma_l(P)$, the group $P$ has nilpotency class at most $c$ and $|P| \leq |\varphi(K)| \leq |Q|$.
	
	Take $s\in \N$ such that $p^{s} \leq r < p^{s+1}$. Now, $\psi(g^{p^s}) = \psi(g)^{p^s} \neq e$ and thus $\gamma_l(P)$ has exponent $\geq p^{s+1}$. By construction, the conditions of Lemma \ref{lem_bound_p_group} above are satisfied, showing that $|Q| \geq |P| \geq p^{(s+1)(l+1)} > r^{l+1}$.
\end{proof}
	The conditions of Theorem \ref{prop_lower_addendum} are clearly satisfied if $G$ has a subgroup $\tilde{G}$ such that $\tilde{G}\cap K = \gamma_l(K)$ and $\tilde{G}$ is not virtually nilpotent.
	\begin{cor} \label{cor_lower_addendum}
		Let $G$ be a $\mathcal{M}$-group with torsion-free nilpotent subgroup $K$ as introduced in Notation \ref{nota_group}. If there exists a subgroup $\tilde{G}$ of $G$ such that $\{e\} \neq \tilde{G}\cap K \leq \gamma_l(K)$ with $l>1$ and $\tilde{G}$ is not virtually nilpotent, then $r^{l+1} \preceq \RF_G$.
	\end{cor}
	\begin{proof}
		Since $\tilde{G}$ is not virtually nilpotent, the group $\tilde{G}\cap \bar{G} \leq_f \tilde{G}$ is not either. Hence, we may suppose that $\tilde{G} \leq \bar{G}$. By assumption on $\tilde{G}$, we know that $\tilde{G}$ has exponential word growth. Using the projection to $\Z^n$ and exactly as in Lemma \ref{lem_pigeon_Zn}, this implies that $|B_{\tilde{G}}(r) \cap (\tilde{G}\cap K)| \leq |B_{\tilde{G}}(r) \cap \gamma_l(K)| $ grows exponentially in $r$. Hence, $|B_{G}(r) \cap \gamma_l(K)|$ also grows exponentially. Now apply Theorem \ref{prop_lower_addendum}.
	\end{proof}
	In the lower bound estimates stated so far, we only made estimates for $\varphi(K)$, where $\varphi: G \to Q$ is a homomorphism to a finite group. In other words, if $N$ is a finite index subgroup of $G$, then we found bounds for $[K: K\cap N]$. However, if $H$ is infinite, $\pi(N)$ also needs to be a finite-index (and thus infinite) subgroup of $H$ (with $\pi: G \to H$). We end this paper with some observations concerning this fact.
	\begin{lemma} \label{lem_lower_extra_log}
		Let $G$ be a group of the form $\Z^m\rtimes_\varphi \Z$, where $\varphi(1) = M \in \GL(m, \Z)$. If $M$ has no eigenvalues that are roots of unity, then $r\log(r) \preceq \RF_G$.  
	\end{lemma}
	\begin{proof}
		Note that $G$ is virtually nilpotent if and only if the eigenvalues of $M$ are roots of unity by \cite[Proposition 4.4.(3.)]{Wolf}. Therefore, $G$ is not virtually nilpotent and thus has exponential word growth, so by Lemma \ref{lem_pigeon_Zn} and the arguments as in the proof of Theorem \ref{thm_lower}, we can take $g^{\lcm(1, 2 , \ldots, r)} \in B_G(Cr)\cap \Z^m$, where $C>0$ is independent of $r\geq 1$.
		
		Let $N$ denote a normal subgroup of $G$ realizing $D_G(g^{\lcm(1, 2 , \ldots, r)})$. We know that $[G:N] = [\Z^m:\Z^m\cap N]\cdot [\Z: \pi(N)]$. Suppose $\pi(N) = l\Z$ with $l\in\N$. For $w\in\Z^m$ arbitrary, we have 
		\begin{equation} \label{eq_extra_log}
			\forall (v,-l)\in N: [(w,0),(v,-l)] = (M^{l}w-w,0) \in N.
		\end{equation}
		Hence, $(M^{l}-\mathbb{1})\Z^n \leq N\cap K$ and $g^{\lcm(1,2, \ldots , r)} \notin (M^{l}-\mathbb{1})\Z^n$. 
		
		Note that $|\det(M^{l}-\mathbb{1})| \Z^n \leq (M^{l}-\mathbb{1})\Z^n$. The condition $g^{\lcm(1,2, \ldots , r)}\notin |\det(M^{l}-\mathbb{1})| \Z^n$ implies that $r< |\det(M^{l}-\mathbb{1})|$.  Since $|\det(M^{l}-\mathbb{1})|$ grows exponentially in $l\in\N$ by the assumption, we conclude that $\log(r) \preceq l = [\Z: \pi(N)]$. Also, $g^{\lcm(1,2, \ldots , r)}\notin N\cap K$ implies that $r \leq [K:N\cap K]$, so 
		$$r\log(r) \preceq [G:N] = D_G(g^{\lcm(1,2, \ldots , r)}) \leq \RF_G(Cr).$$
	\end{proof}
	\begin{ex}
		If $G$ is a group of the form $\Z^m\rtimes_\varphi \Z$ where $\varphi(1)$ has at least one eigenvalue that is not a root of unity, then it always contains a subgroup that satisfies the condition of the lemma above. Hence, the lower bound $r\log(r)$ holds for all groups of the given form that are not virtually nilpotent.
	\end{ex}
	\begin{ex}
		This bound also applies to the Baumslag-Solitar group $\BS(1,n) \cong \Z[1/n]\rtimes \Z$ with $|n|>1$  as in Example \ref{ex_BS}. Indeed, just as in the proof above, the condition on $l\Z = \pi(N) \lhd \Z$ given in Equation \eqref{eq_extra_log} becomes $(n^{l}-1)\Z[1/n] \leq N\cap K$. Therefore, $r<n^l-1$ and $r\log(r) \preceq \RF_{\BS(1,n)}$. Recall that Theorem \ref{thm_upper} states that $\RF_{\BS(1,n)} \preceq r^2$, but the exact function remains unknown.
	\end{ex}
	
	Note that the bound $r\log(r) \preceq \RF_{G}$ applies to all $\mathcal{M}$-groups having $\BS(1,n)$ with $|n|>1$ or $\Z^m\rtimes_\varphi \Z$ as in Lemma \ref{lem_lower_extra_log} as a subgroup. One can ask whether this lower bound always holds:
	\begin{ques}
		Is it true that $r\log(r) \preceq \RF_G$ holds for all $\mathcal{M}$-groups that are not virtually nilpotent?
	\end{ques}
	A positive answer would also raise the question whether results like Theorem \ref{prop_lower_addendum} can be generalized to obtain bounds of the form $r^l\log(r) \preceq \RF_G$ or better.

\bibliographystyle{plain}
\bibliography{Dere_Matthys_minimax}
\end{document}